\documentclass[reqno,11pt]{amsart}
\usepackage{amsmath,amssymb,mathrsfs,amsthm,amsfonts}
\usepackage[latin1]{inputenc}   
\usepackage[inline]{enumitem}
\usepackage[usenames,dvipsnames]{xcolor}
\usepackage{hyperref}
\usepackage[percent]{overpic}
\usepackage{comment}
\usepackage{algorithm}
\usepackage{algorithmic}
\usepackage{mathtools}
\usepackage{subcaption}
\usepackage{tikz}
\usepackage{bm}
\usetikzlibrary{calc}
\usetikzlibrary{arrows.meta,backgrounds}
\usepackage{multirow,array,longtable,booktabs}
\usetikzlibrary{arrows}
\newcommand{\C}{{\mathcal C}}

\newtheorem{theorem}{Theorem}[section]
\newtheorem{lemma}[theorem]{Lemma}
\newtheorem{proposition}[theorem]{Proposition}

\newtheorem{prop}[theorem]{Proposition}

\theoremstyle{definition}
\newtheorem{definition}[theorem]{Definition}

\newtheorem{remark}[theorem]{Remark}
\newtheorem{remarks}[theorem]{Remarks}
\numberwithin{equation}{section}
\allowdisplaybreaks
\usepackage{acronym}

\acrodef{LDP}{Large Deviation Principle}

\newcommand{\F}{\mathbb{F}}
\renewcommand{\P}{\mathbb{P}}	
\newcommand{\Q}{\mathbb{Q}}
\newcommand{\E}{\mathbb{E}}	

\newcommand{\be}{\begin{equation}}
\newcommand{\ee}{\end{equation}}
\newcommand{\bea} {\begin{array}{rl}}
\newcommand{\eea} {\end{array}}
\newcommand{\bepa}{\left\{ \begin{array}{l}}
\newcommand{\eepa} {\end{array}\right.}
\makeatletter

\newcommand{\R}{\mathbb{R}} 
\newcommand{\N}{\mathbb{N}}

\newcommand{\ep}{\epsilon}

\newcommand{\ds}{\displaystyle}

\renewcommand{\hat}{\widehat}

\renewcommand{\bar}{\overline}

\usepackage{graphicx}

\newcommand{\bx}{\bm{x}}
\newcommand{\bw}{\bm{w}}

\newcommand{\eps}{\epsilon}

\newcommand{\by}{\bm{y}}

\title[Regularization of a mean-field SDE by a common noise]{Regularization of a mean-field SDE by an additive common noise: the conditional expectation case}

\author[Pierre Cardaliaguet \and Benjamin Jourdain]{Pierre Cardaliaguet$^\ast$ \and Benjamin Jourdain$^\dagger$}
\address{$^\ast$Ceremade (UMR CNRS 7534), Universite Paris-Dauphine and PSL,
Place du Maréchal De Lattre De Tassigny, 75775 Paris Cedex 16,  France. Email: cardaliaguet@ceremade.dauphine.fr}
\address{$^\dagger$CERMICS, CNRS, INRIA, ENPC, Institut Polytechnique de Paris, Marne-la-Vallée, France. Email:benjamin.jourdain@enpc.fr}
\thanks{P. Cardaliaguet was partially supported by the Agence Nationale de la Recherche (ANR), project ANR-22-CE40-0010 COSS. The work was completed during the period when the first author was hosted by the INRIA project team Martingale, that he wishes to thank warmly.}
\usepackage[foot]{amsaddr}
\begin{document}
	
\begin{abstract} We investigate a McKean-Vlasov stochastic differential equation with an additive common noise and in which the interaction is  through the conditional expectation. We show that, in the presence of an additive individual noise, existence and uniqueness of a weak solution hold for any drift given by a bounded and measurable function of the position and the conditional expectation. When there is no individual noise, existence and uniqueness still hold if the drift is in addition Lipschitz in the position variable.  This shows that the presence of a finite dimensional common noise may allow to overcome the discontinuity of the drift with respect to the interaction term, provided that this interaction term is a conditional expectation. We also prove propagation of chaos for systems of particles where the conditional expectation is replaced by the empirical mean of the positions or by a closely related contribution with better prepared noise.
\end{abstract}
	
	
	\maketitle


\section{Introduction}
We investigate the  following McKean-Vlasov stochastic differential equation with a common noise:
\be\label{eq.barXsigma=0}
X_t =  X_0+ \int_0^t  b(X_s, \E\left[X_s|W^0\right]) ds +\sigma W_t+ \sigma^0dW^0_t,\; t\in[0,T],
\ee
where $T\in (0,+\infty)$ is a finite time-horizon.
The problem is formally set on a filtered probability space $(\Omega, \mathcal F, P, (\mathcal F_t))$ endowed with two independent $d-$dimensional Brownian motions $W$ and $W^0$; $W$ is often called the individual (or idiosyncratic) noise; $W^0$ is the common noise. The initial condition $ X_0$ is an $\R^d$-valued random vector with law $\mu_0$ and is independent of $W$ and $W^0$. The drift coefficient $b:\R^d\times \R^d\to \R^d$ is (at least) Borel measurable and bounded, while the real numbers $\sigma$ and $\sigma^0$ are assumed to be nonnegative. Our goal is to establish the existence and the uniqueness of the weak solution of this equation in two settings: 
\begin{enumerate}
\item[(A)] either $\sigma=0$ and $\sigma^0$ is positive, in which case we suppose in addition that $b$ is Lipschitz continuous in the first variable uniformly in the second one,

\item[(B)] or $\sigma$ and $\sigma^0$ are positive, in which case no additional condition is supposed on $b$. 
\end{enumerate}
We also address the particle approximation of the solution in both cases.  \\ 

Equation \eqref{eq.barXsigma=0} is a typical example of McKean-Vlasov diffusion with a common noise. This kind of model has attracted a lot of attention in the recent years,  largely in connection with questions in Mean Field Game theory, see \cite{CDLL, CDbook}.  When the drift coefficient $b$ is globally Lipschitz and depends more generally on the conditional law of the process with respect to the common noise, the existence and the uniqueness of a strong solution to this equation is well-known: this has been proved by Sznitman \cite{Sz06} for problems without common noise (i.e., $\sigma^0=0$) where the conditional law amounts to the regular law, and by Vaillancourt \cite{Va88}, Dawson and Vaillancourt \cite{DaVa95}  and Goghi and Gess \cite{CoGe19} in the presence of common noise ($\sigma^0>0$). In the presence of common noise, the convergence of the associated particle system has been addressed by Kurtz and Xiong \cite{KuXi99} and by Coghi and Flandoli \cite{CoFl16} (see also the monograph by Carmona and Delarue \cite{CDbook}). The Lipschitz regularity assumption on the space variable can  be largely relaxed in the presence of individual noise ($\sigma>0$). This has been documented when there is no common noise: see for instance \cite{Jo97}  by the second author,  \cite{JaWa16, JaWa18} by Jabin and Wang,  \cite{La18}  by Lacker, \cite{Jab19} by Jabir or  \cite{MiVe20} by Mishura and Veretennikov, to quote only a few references. Recent works also investigate the case where there is a common noise, as Hammersley, Siska and Szpruch \cite{HSS}, Crowell \cite{Cr24, Cr25} or Shkolnikov and Yeung \cite{ShYe26}. However little is understood when the drift $b$ is discontinuous with respect to the (conditional) law of the process: it is then easy to find examples where the existence or uniqueness of the solution is called into question, even for a drift $b$ which is Lipschitz in the first variable: see Appendix \ref{sec.appen}, where the nonlocal term is a conditional variance instead of the conditional expectation. 

Our aim is to explain that the presence of the common noise allows to overcome the discontinuity of the drift with respect to the interaction term, provided that this interaction term is a {\it conditional expectation}. In the framework of Mean Field Game (MFG) theory, the regularizing property of the (finite dimensional) common noise has been observed by Delarue and his co-authors  in several instances \cite{DM25, DeTc20, DeVa25}. In MFG,  the presence of the common noise ensures the uniqueness of the MFG equilibrium, while this uniqueness might not hold without common noise. Another property of the common noise, studied by Maillet \cite{Ma23} and by Delarue, Maillet and Tanr\'e  \cite{DMT24}, is that it ensures the uniqueness of the invariant measure (understood in a suitable sense) for some McKean-Vlasov equations. 

Our main results are the existence and the uniqueness of a weak solution (understood in a suitable sense, see Definitions \ref{def.weaksolNIN} and \ref{def.weaksolNINCN}) under condition (A) or condition (B).  Roughly speaking, the proof relies on the remark that the quantity $Y= X-\E[X|W^0]$ formally satisfies a ``regular'' McKean-Vlasov equation (parametrized by the process $X^0= \E[X|W^0]$), while the equation  for $X^0= \E[X|W^0]$, once $Y$ is built, becomes a SDE with an additive noise (the common noise) and a bounded drift, which can be solved by a Girsanov transform. This program is achieved in Theorem \ref{thm.existenceNIN} for case (A) (without individual noise) and  in Theorem \ref{thm.existenceCN} for the  more delicate case  (B) (with individual noise). We also prove the convergence of  associated particle systems. In case (A), the ``natural'' system takes the form  of the SDE
$$
X^{i,N}_t  = X^{i}_0 + \int_0^t b(X^{i,N}_s, \frac{1}{N}\sum_{j=1}^N X^{j,N}_s)ds + \sigma^0 W^0_t, 
$$
where the initial conditions $(X^{i}_0)$ are i.i.d. according to $\mu_0$ and independent from $W^0$. The convergence (Theorem \ref{propchaosbissansbruitindiv}) then relies on the continuity of the law of the solution to \eqref{eq.barXsigma=0} with respect to the initial measure (Proposition \ref{prop.continuity}).  The convergence rate we obtain in this setting is quite slow. If we look instead at  the ``well-prepared'' particle system 
\begin{align*}
   & X^{i,N}_t=X^i_0+\int_0^t b(X^{i,N}_s,X^{0,N}_s)ds+\sigma^0 W^0_t\mbox{ with }\\
   & X^{0,N}_s= \frac 1N\sum_{j=1}^N X^{j,N}_s+ \E[X^1_0]- \frac 1N\sum_{j=1}^NX^j_0.
\end{align*}
the convergence is more direct and the convergence rate much sharper (Theorem \ref{thm.cvparticleNOIN}). The key idea is that, in this new system, the second variable $X^{0,N}$ in the drift coefficient has the same initial condition as $X^0$ in the limit equation, which allows one to go from the particle system to the McKean-Vlasov equation by Girsanov's theorem. The well-prepared particle system in which $X^{0,N}_s$ differs from $\frac 1N\sum_{j=1}^N X^{j,N}_s$ by some centred contribution can be interpreted as a  control variate variance reduction method.

In case (B), the convergence of the particle system turns out to be much more challenging and we only manage in Theorem \ref{thm.cvparticleWIN} to prove it for the ``well-prepared'' particle system: 
\begin{align*}
& X^{i,N}_t=  X^i_0+ \int_0^t b\Bigl(X^{i,N}_s,X^{0,N}_s\Bigr)ds + \sigma W^i_t+\sigma^0 W^0_t, \\
& X^{0,N}_s = \frac{1}{N} \sum_{j=1}^N X^{j,N}_t+ \bar \mu_0-\frac{1}{N} \sum_{j=1}^N X^{j}_0- \frac{\sigma}{N} \sum_{j=1}^N W^i_t
\end{align*}
Note that with this choice, $X^{0,N}$ has the same initial condition and the same additive noise as $X^0=\E[X|W^0]$ in the limit system. In this regime, we follow the approach by Jabir \cite{Jab19} to obtain an estimate of order $((1+k)/N)^{1/2}$ for the total variation distance between the law of $(X^{0,N},X^{1,N},\cdots,X^{k,N})$  in the above system and the law of $X^0=\E[X|W^0]$ together with $k$ conditionally independent copies distributed according to the conditional law of $X$ given $X^0$. 

We conclude the introduction with a few open problems. The first open question concerns the existence of strong solutions for \eqref{eq.barXsigma=0}: indeed, we only build a weak solution, in which the conditional expectation is adapted to a  filtration possibly larger than the one generated by $W^0$. A second open question is the convergence of the ``natural'' particle system in the presence of the individual noise (see also Remark \ref{rem.propag}). A last very intriguing open problem is to find other  structure conditions ensuring that the common noise is regularizing for a general conditional McKean-Vlasov equation of the form 
$$
X_t =  X_0+ \int_0^t  b(X_s, \mathcal L\left[X_s|W^0\right]) ds +\sigma W_t+ \sigma^0dW^0_t.
$$
Indeed, we use that the interaction is through the conditional expectation in \eqref{eq.barXsigma=0} in a crucial way throughout the paper. An example in the appendix (where the interaction is through the conditional variance) shows that existence may fail without some structure conditions. 
\bigskip 

{\it Organization of the paper:} In Section \ref{sec.1}, we discuss the problem without individual noise (i.e., Case A): we show the existence and uniqueness of the solution (Subsection \ref{subsec2.1}), the continuity with respect to the initial measure (Subsection \ref{subsec2.2}) and the propagation of chaos (Subsection \ref{subsec2.3}). Section \ref{sec.3} is dedicated to the problem with both individual and common noises (Case B): the existence and the uniqueness of the solution are obtained in Subsection \ref{subsec3.1}, while the propagation  of chaos is proved of Subsection \ref{subsec3.2}. 
\bigskip

{\it Notation:} 
\begin{itemize}
  \item Given a metric space $E$, we denote by ${\mathcal P}(E)$ the set of Borel probability measures on $E$.
\item We denote by $\mathcal P_1(\R^d)$ the set of Borel probability measures on $\R^d$ with finite first order moment. It is endowed with the Wasserstein distance $W_1$. 
For $\mu_0\in{\mathcal P}_1(\R^d)$, we set $\bar\mu_0=\int_{\R^d}x\mu_0(dx)$. 
\item Given a measure $\mu$ on a set $E$ and a measurable map $\phi:E\to F$, we denote by $\phi\sharp \mu$ the image of $\mu$ by $\phi$. 
\item We denote by ${\mathcal C}=C([0,T], \R^d)$ the space of continuous paths from $[0,T]$ to $\R^d$.
\item We denote by ${\mathcal W}$  the Wiener measure on ${\mathcal C}$ i.e. the distribution of a $d$-dimensional Brownian motion $(B_t)_{t\in[0,T]}$ and, for $x\in\R^d$, by ${\mathcal W}^{x}$ the Wiener measure on ${\mathcal C}$ with initial condition $x$, i.e. the distribution of $(x +B_t)_{t\in[0,T]}$.
\item Given a set $E$ and a real valued map $f$ on $E$, $\|f\|_\infty$ is the supremum of $f$ on $E$, or the ess-sup of $f$ on $E$, depending on the context. 
\item Given a random vector $X$, $\mathcal L(X)$ stands for the law of $X$. 
\item $Id$ is the identity map on $\R^d$. 
\end{itemize}

\section{The problem with no individual noise} \label{sec.1}

We first consider Equation \eqref{eq.barXsigma=0} when there is no individual noise: $\sigma=0$. We look for a solution in the form 
\be\label{eq.barXsigma=0BIS}
\left\{\begin{array}{l}
\ds X_t =  X_0+ \int_0^t b(X_s, X^0_s) dt + \sigma^0 W^0_t, \qquad t\in [0,T], \\
\ds X^0_t= \E\left[X_t\ |\ W^0\right], \quad t\in [0,T], \qquad \mathcal L( X_0)= \mu_0. 
\end{array}\right.
\ee
Throughout this section we suppose that 
\be\label{Hyp.noCN} \begin{array}{c}
\mbox{\rm $\sigma^0$ is positive, $\mu_0\in \mathcal P_1(\R^d)$ and $b:\R^d\times \R^d\to \R^d$ is Borel measurable, bounded,}\\ 
\mbox{\rm and Lipschitz continuous with respect to the first variable,}\\
\mbox{\rm uniformly in the second variable.} 
\end{array} \ee

\begin{definition}[Weak solution] \label{def.weaksolNIN} 
A weak solution to \eqref{eq.barXsigma=0BIS} consists of a tuple $(\Omega, (\mathcal F_t), \F,{\mathbb P}, W^0,X^0,X)$, where  $(\Omega, (\mathcal F_t), \F,{\mathbb P})$ is a filtered probability space with a complete filtration supporting $(W^0,X^0,X)$, such that  
\begin{itemize}
\item[(i)] $W^0$ is a $d-$dimensional $(\mathcal F_t)-$Brownian motion,  the processes $X$ and $X^0$ are $(\mathcal F_t)-$adapted with values in $\R^d$, 

\item[(ii)] $X_0$ is independent of $(W^0, X^0)$ and distributed according to $\mu_0$, 

\item[(iii)] The state equation holds: 
$$
 X_t =  X_0+ \int_0^t b(X_s, X^0_s) ds +\sigma^0 W^0_t.
 $$

\item[(iv)]  $W^0$ is adapted to the filtration generated by $X^0$, and  $X^0_t= \E\left[X_t\ |\ X^0\right]$ for any $t\in [0,T]$. 
\end{itemize}
\end{definition}

The notion of weak solution is inspired by the notion of equilibrium solution  for mean field games  with a common noise in the paper by Carmona, Delarue and Lacker \cite{CaDeLa16}. We note that condition (iv) differs from the equality $X^0_t= \E\left[X_t\ |\ W^0\right]$ in \eqref{eq.barXsigma=0BIS} and is slightly more precise than the equality $X^0_t= \E\left[X_t\ |\ W^0,X^0\right]$ in \cite{CaDeLa16}:  this  sharpened condition is needed to ensure the uniqueness of the solution. Let us finally underline that we do not know if there exists a strong solution, i.e., a solution for which the filtrations of $X^0$ and $W^0$ coincide. 

\subsection{Existence and uniqueness of a  solution} \label{subsec2.1}

\begin{theorem}\label{thm.existenceNIN} Under Assumption \eqref{Hyp.noCN}, there exists a weak solution to \eqref{eq.barXsigma=0BIS}. In addition, this solution is unique in law. 
\end{theorem}

The key step towards the proof  is the following intermediate result, the proof of which is postponed after the one of the theorem. This result introduces the map ${\mathcal Y}^{\mu_0}$, which will allow us to make the distribution of the unique solution to \eqref{eq.barXsigma=0BIS} explicit. 
 Let us recall that $\mathcal C$ is the set of continuous maps from $[0,T]$ to $\R^d$ and that, given $\mu_0\in \mathcal P_1(\R^d)$, $\ds \bar \mu_0= \int_{\R^d} x\mu_0(dx)$. 
 
 \begin{lemma}\label{lem.mathcalY} For any $\mu_0\in \mathcal P_1(\R^d)$, there exists a unique Borel measurable map $\mathcal Y^{\mu_0}: [0,T]\times \R^d\times  {\mathcal C}\to \R^d$ such that, for any $(t,x,\bx)\in [0,T]\times \R^d\times {\mathcal C}$, 
$$
\mathcal Y^{\mu_0}(t,x, {\bx}) = x-\bar \mu_0+\int_0^t  b(\mathcal Y^{\mu_0}(s,x, {\bx})+{\bx}_s, {\bx}_s)ds - \int_0^t\int_{\R^d} b(\mathcal Y^{\mu_0}(s,y, {\bx})+{\bx}_s, {\bx}_s)\mu_0(dy)ds. 
$$
In addition, $\mathcal Y^{\mu_0}$ is nonanticipating in the sense that 
$$
\mathcal Y^{\mu_0}(t,x, {\bx}) = \mathcal Y^{\mu_0}(t,x, {\bx}(\cdot \wedge t))
$$
and the map $(x,\mu_0)\to \mathcal Y^{\mu_0}(t,x,\bx)$ is Lipschitz continuous on $\R^d\times \mathcal P_1(\R^d)$, uniformly with respect to $t\in [0,T]$ and $\bx \in {\mathcal C}$. Last, if $(y^1_t,\cdots,y^N_t)_{t\in[0,T]}$ solves 
\begin{equation}
   \forall t\in[0,T],\;y^i_t=x^i_0-\frac 1N\sum_{j=1}^Nx^j_0+\int_0^t  b(y^i_s+{\bx}_s, {\bx}_s)ds - \int_0^t\frac 1N\sum_{j=1}^Nb(y^j_s+{\bx}_s, {\bx}_s)ds,\;i\in\{1,\cdots,N\},\label{eq:discdyn}
\end{equation}
then 
$$
(y^1_t,\cdots,y^N_t)_{t\in[0,T]}=(\mathcal Y^{m^N_{\vec{x}_0}}(t,x^1_0, {\bx}),\cdots,\mathcal Y^{m^N_{\vec{x}_0}}(t,x^N_0, {\bx}))_{t\in[0,T]},
$$ 
where we recall the notation $m^N_{\vec{x}_0}=\frac 1N\sum_{i=1}^N\delta_{x^i_0}$. 
\end{lemma}
For $\mu_0\in{\mathcal P}_1(\R^d)$, $t\in[0,T]$, $z\in\R^d$ and $\bx\in{\mathcal C}$, let us set 
\be\label{defbarb}
\bar b^{\mu_0}(t,z,{\bx}) = \int_{\R^d} b(\mathcal Y^{\mu_0}(t, y, {\bx})+z,z) \mu_0(dy). 
\ee
The function $\bar b^{\mu_0}$ is bounded and inherits the nonanticipating property of $\mathcal Y^{\mu_0}$ stated in Lemma \ref{lem.mathcalY}. \begin{remark}\label{rem:PQ}
We show at the end of the proof of Theorem \ref{thm.existenceNIN} that the distribution $P^{\mu_0}(d{\bw}^0,d{\bx}^0,d{\bx})$ of the weak solution $(W^0, X, X^0)$ of \eqref{eq.barXsigma=0BIS} is the image of the measure 
$$
\mu_0(dx_0)\exp\left\{ 
\int_0^T \bar b^{\mu_0}(s, {\bx}^0_s, {\bx}^0)\cdot d{\bx}^0_s 
- \frac12 \int_0^T |\bar b^{\mu_0}(s, {\bx}^0_s, {\bx}^0)|^2ds 
\right\}{\mathcal W}^{\bar\mu_0}(d{\bx}^0)
$$ 
by the map 
$$
\R^d\times {\mathcal C}\mapsto \left({\bx}^0-\bar\mu_0-\int_0^\cdot\bar b^{\mu_0}(s,{\bx}^0_s,{\bx}^0)ds,{\bx}^0,\mathcal Y^{\mu_0}(\cdot,x_0, {\bx}^0)+{\bx}^0\right)\in {\mathcal C}^3.
$$
Note that the marginal distribution $Q^{\mu_0}(d{\bx}^0)=\int_{\bw^0,\bx}P^{\mu_0}(d{\bw}^0,d{\bx}^0,d{\bx})$  determines $P^{\mu_0}$. Indeed $P^{\mu_0}$ (resp. the distribution $\int_{\bx\in\C}P^{\mu_0}(d{\bw}^0,d{\bx}^0,d{\bx})$ of the $(W^0,X^0)$ components) is the image of $\mu_0(dx_0)Q^{\mu_0}(d\bx^0)$ (resp. $Q^{\mu_0}$) by a deterministic function depending on the initial distribution $\mu_0$ through $(\bar b^{\mu_0},{\mathcal Y}^{\mu_0})$ (resp. $\bar b^{\mu_0}$).
\end{remark}

To simplify the notation, we assume from now on and without loss of generality that $\sigma^0=1$. 

\begin{proof}[Proof of Theorem \ref{thm.existenceNIN}] 
  Let us  fix a probability space $(\Omega, \mathcal F, \F, {\mathbb Q})$ supporting a standard Brownian motion $B$ and an $\F_0-$measurable random variable $X_0$ with law $\mu_0$ independent of $B$. Note that  $\bar \mu_0=\E[X_0]$. Recalling that $\bar b^{\mu_0}$, introduced in \eqref{defbarb}, is bounded and inherits the nonanticipating property of $\mathcal Y^{\mu_0}$ stated in Lemma \ref{lem.mathcalY}, we can define the probability measure ${\mathbb P}$ equivalent to ${\mathbb Q}$  by 
\be\label{defPP}
\frac{d{\mathbb P}}{d{\mathbb Q}} = \exp\left\{ \int_0^T \bar b^{\mu_0}(s, \bar \mu_0+B_s, \bar \mu_0 +B)\cdot dB_s - \frac12 \int_0^T |\bar b^{\mu_0}(s, \bar \mu_0+B_s,\bar \mu_0+ B)|^2ds \right\}. 
\ee
By Girsanov Theorem,  under ${\mathbb P}$, the process
\be\label{defbeta}
\left(W^0_t = B_t -\int_0^t \bar b^{\mu_0}(s, \bar \mu_0+B_s, \bar \mu_0+B)ds\right)_{t\in [0,T]}
\ee
is a Brownian motion independent of $X_0$ which is still distributed according to $\mu_0$. We also set 
$$
X_t:= \mathcal Y^{\mu_0}(t, X_0, \bar \mu_0+B)+\bar \mu_0+B_t, \qquad X^0_t= \bar \mu_0+ B_t.
$$
Our aim is to prove that $(\Omega,  \mathcal F, \F, {\mathbb P}, X, X^0)$ is a weak solution of  \eqref{eq.barXsigma=0BIS}. For this we first note that  
\begin{align*}
\E_{{\mathbb P}}\left[ \mathcal Y^{\mu_0}(t, X_0, \bar \mu_0+B)\ |\ B\right] & = \E_{\mathbb P}[X_0\ |\ B] -\bar \mu_0 + \int_0^t\Big( \E_{\mathbb P}\left[ b(\mathcal Y^{\mu_0}(s, X_0, \bar \mu_0+B)+\bar \mu_0+B_s,\bar \mu_0+B_s)\ |\ B\right]\\&-\int_{\R^d}b(\mathcal Y^{\mu_0}(t, y, \bar \mu_0+B)+\bar \mu_0+B_s,\bar \mu_0+B_s)\mu_0(dy)\Big)ds =0, 
\end{align*}
since $X_0$ is independent of $B$ (even under ${\mathbb P}$) and has law $\mu_0$. Thus
\be\label{relbarXB}
X^0_t =\bar \mu_0+ B_t= \E_{\mathbb P}\left[\mathcal Y^{\mu_0}(t, X_0, \bar \mu_0+B)+\bar \mu_0+ B_t |\ B\right]= \E_{\mathbb P}\left[ X_t\ |\ B\right].
\ee
Therefore, from the definition of $X$ and $X^0$,  
$$
b(\mathcal Y^{\mu_0}(s,  X_0, \bar \mu_0+B)+ \bar \mu_0 +B_s,\bar \mu_0 +B_s)= b(X_s,X^0_s)
$$
and 
$$
\int_{\R^d} b(\mathcal Y^{\mu_0}(s,  y, \bar \mu_0+B)+ \bar \mu_0 +B_s,\bar \mu_0 +B_s)\mu_0(dy) = \E_{\mathbb P}\left[ b(X_s,X^0_s)\ |\ B\right].
$$
We infer from the definition of $X$, the equation satisfied by $\mathcal Y^{\mu_0}$, \eqref{defbarb} and the definition of $W^0$, that 
\begin{align*}
X_t &= \mathcal Y^{\mu_0}(t, X_0-\bar \mu_0, \bar \mu_0+B)+\bar \mu_0+B_t\\
&  = X_0-\bar \mu_0 +\int_0^t b(X_s,X^0_s)ds -\int_0^t \bar b^{\mu_0}(s, \bar \mu_0+B_s, \bar \mu_0+B)ds
  \\
& \qquad + \bar \mu_0 + W^0_t + \int_0^t \bar b^{\mu_0}(s, \bar \mu_0+B_s, \bar \mu_0+B)ds \\ 
&= X_0 +\int_0^t b(X_s,X^0_s)ds+ W^0_t.
\end{align*}
To complete the proof we note that, by the first equality in \eqref{relbarXB}, $\mathcal F^{X^0}=\mathcal F^B$  and therefore  
$$
\E\left[ X_t\ |\ X^0\right]= \E\left[ X_t\ |\ B\right]=X^0_t.
$$
This proves that $(\Omega, \mathcal F, \F, {\mathbb P}, W^0,X^0,X)$ satisfies all the conditions of Definition \ref{def.weaksolNIN}. \\

Let us now prove the uniqueness of the process.  Let $(\Omega, (\mathcal F_t), \F,{\mathbb P}, W^0,X^0,X)$ be a weak solution to \eqref{eq.barXsigma=0BIS} in the sense of Definition \ref{def.weaksolNIN}. We first claim that for $\mathcal Y^{\mu_0}$ defined in Lemma \ref{lem.mathcalY} ${\mathbb P}-$a.s. and for any $t\in [0,T]$, 
$$
Y_t:= X_t-X^0_t= \mathcal Y^{\mu_0}(t, X_0, X^0), 
$$ 
which, with Definition \ref{def.weaksolNIN} (iii)-(iv), implies \begin{equation}
   X^0_t=\E[X_t|X^0]=\bar\mu_0+\int_0^t \E[b(Y_s+X^0_s,X^0_s)|X^0]ds+W^0_t.\label{eq;x0}
\end{equation}Indeed, $Y$ satisfies, for any $t\in [0,T]$: 
$$
Y_t= X_0-\bar \mu_0 +\int_0^t b(Y_s+X^0_s, X^0_s)ds -\int_0^t\E\left[ b(Y_s+X^0_s, X^0_s)\ |\ X^0\right]ds .
$$
Let $\tilde Y_t= \mathcal Y^{\mu_0}(t, X_0, X^0)$. Then 
\begin{align*}
&\E\left[ b(\tilde Y_s+X^0_s, X^0_s)\ |\ X^0\right]  =
\E\left[ b(\mathcal Y^{\mu_0}(s, X_0, X^0) +X^0_s, X^0_s)\ |\ X^0\right] \\
&\qquad =  \int_{\R^d} b(\mathcal Y^{\mu_0}(s, y, X^0)+X^0_s, X^0_s) \mu_0(dy),
\end{align*}
since $X_0$ is independent of $X^0$ and is distributed according to $\mu_0$. 
Thus, $\tilde Y$ satisfies 
$$
\tilde Y_t = X_0-\bar \mu_0 +\int_0^t b(\tilde Y_s+ X^0_s  ,X^0_s)ds - \int_0^t \E\left[ b(\tilde Y_s+X^0_s, X^0_s)\ |\ X^0\right] ds. 
$$
Denoting by $K$ the Lipschitz constant of $b$ in its first argument,  we find therefore: 
\begin{align*}
|Y_t-\tilde Y_t| & \leq \int_0^t |b(Y_s+X^0_s,X^0_s)-b(\tilde Y_s+X^0_s,X^0_s)|ds \\
& \qquad + 
\int_0^t \E\left[ |b(Y_s+X^0_s,X^0_s)-b(\tilde Y_s+X^0_s,X^0_s)| \ |  X^0\right]ds  \\
& \leq K \int_0^t|Y_s-\tilde Y_s| ds + K \int_0^t \E\left[ |Y_s-\tilde Y_s|  \ |  X^0\right] ds 
\end{align*}
We take the expectation and infer by Gronwall's lemma that, ${\mathbb P}-$a.s., $Y_t=\tilde Y_t$ for any $t\in [0,T]$. 

With \eqref{eq;x0}, this implies that $X^0$ satisfies
\begin{align}
X^0_t
&= \bar \mu_0 +\int_0^t \int_{\R^d} b( \mathcal Y^{\mu_0}(s, y, X^0)+X^0_s,X^0_s) \mu_0(dy)ds + W^0_t \notag\\ 
& = \bar \mu_0 +\int_0^t  \bar b^{\mu_0}(s, X^0_s,X^0)ds +W^0_t,\label{eq:dynx^0}
\end{align}
where $\bar b^{\mu_0}$ is defined by \eqref{defbarb}.
Let $\tilde {\mathbb P}$ be such that 
\begin{align*}
\frac{d\tilde {\mathbb P}}{d {\mathbb P}} = Z_T&:=\exp\left\{ -\int_0^T \bar b^{\mu_0}(s, X^0_s, X^0)\cdot dW^0_s - \frac12 \int_0^T |\bar b^{\mu_0}(s, X^0_s, X^0)|^2ds \right\} \\
& =\exp\left\{ -\int_0^T \bar b^{\mu_0}(s, X^0_s, X^0)\cdot dX^0_s + \frac12 \int_0^T |\bar b^{\mu_0}(s, X^0_s, X^0)|^2ds \right\}. 
\end{align*}
By Girsanov's theorem,  $(X^0-\bar \mu_0)_{t\in [0,T]}$ is a Brownian motion under $\tilde {\mathbb P}$.  Let $A$ be a Borel subset of ${\mathcal C}^3$. Then 
\begin{align*}
&{\mathbb P}\left[ (W^0,X^0,X)\in A\right]  = \E_{\tilde{\mathbb P}}\left[ {\bf 1}_{A}(X^0-\bar \mu_0 -\int_0^\cdot  \bar b^{\mu_0}(s, X^0_s,X^0)ds,X^0 ,  \mathcal Y^{\mu_0}(\cdot,X_0, X^0)+X^0) Z_T^{-1}\right] \\ 
& = \E_{\tilde{\mathbb P}}\bigg[ {\bf 1}_{ A} \left(X^0-\bar \mu_0 -\int_0^\cdot  \bar b^{\mu_0}(s, X^0_s,X^0)ds ,X^0,  \mathcal Y^{\mu_0}(\cdot,X_0, X^0)+X^0\right)
\\&\phantom{= \E_{\tilde{\mathbb P}}\bigg[ }\times\exp\left\{ 
\int_0^T \bar b^{\mu_0}(s, X^0_s, X^0)\cdot dX^0_s 
- \frac12 \int_0^T |\bar b^{\mu_0}(s, X^0_s, X^0)|^2ds 
\right\}
\bigg] .
\end{align*}
As $X_0$ and $X^0$ are independent under $\tilde {\mathbb P}$, the law of $(X_0, X^0)$ under $\tilde {\mathbb P}$ is the product of $\mu_0$ and ${\mathcal W}^{\bar \mu_0}$. Hence the law of $(W^0,X^0,X)$ is the image of 
$$
\mu_0(dx_0)\exp\left\{ 
\int_0^T \bar b^{\mu_0}(s, {\bx}^0_s, {\bx}^0)\cdot d{\bx}^0_s 
- \frac12 \int_0^T |\bar b^{\mu_0}(s, {\bx}^0_s, {\bx}^0)|^2ds 
\right\}{\mathcal W}^{\bar\mu_0}(d{\bx}^0)
$$ 
by 
$$
\R^d\times {\mathcal C}\ni(x_0,{\bx}^0)\mapsto \left({\bx}^0-\bar\mu_0-\int_0^\cdot\bar b^{\mu_0}(s,{\bx}^0_s,{\bx}^0)ds,{\bx}^0,{\bx}^0+\mathcal Y^{\mu_0}(\cdot,x_0, {\bx}^0)\right)\in{\mathcal C}^3.
$$
\end{proof}
\begin{proof}[Proof of Lemma \ref{lem.mathcalY}] The proof is standard. Fix $({\bx},\mu_0)\in {\mathcal C}\times \mathcal P_1(\R^d)$. Let $E$ be the Banach space of  continuous maps $z: [0,T]\times \R^d \to \R^d$ with at most a linear growth, endowed with the norm 
$$
\|z\| = \sup_{(t,x)} e^{-kt} \frac{|z(t,x)|}{1+|x|}, 
$$
for a constant $k>0$ to be chosen below. We define the map $\Phi:E\to E$ by 
$$
\Phi(z)(t,x)= x-\bar \mu_0 + \int_0^t b(z(s,x)+{\bx}_s, {\bx}_s)ds - \int_0^t\int_{\R^d} b(z(s,y)+{\bx}_s, {\bx}_s)\mu_0(dy)ds, 
$$
for any $z\in Z$. 
Then, denoting by $K$ the Lipschitz constant of $b$ in its first argument,  we find, for any $z_1,z_2\in E$,
$$
\frac{| (\Phi(z_1)-\Phi(z_2))(t,x)|}{1+|x|} =  2K \|z_1-z_2\|\int_0^t e^{ks}ds  \leq 2Kk^{-1} e^{kt}\|z_1-z_2\|. 
$$
Therefore 
$$
\|\Phi(z_1)-\Phi(z_2)\| \leq 2Kk^{-1} \|z_1-z_2\|, 
$$
which shows that $\Phi$ is a contraction for $k>2K$. Thus the map $(t,x)\to \mathcal Y^{\mu_0}(t,x, {\bx}) $ is well-defined for any $({\bx},\mu_0)\in {\mathcal C}\times \mathcal P_1(\R^d)$. It is Borel measurable with respect to $\bx$ because the map $\Phi$ above is and $\mathcal Y^{\mu_0}(t,x, {\bx}) $ can be obtained by iterating this map. 

For the regularity of $\mathcal Y^{\mu_0}$ with respect to the second variable, we note that, for any $x,x'\in \R^d$ (letting $z(t,x)= \mathcal Y^{\mu_0}(t,x,\bx)$), 
\begin{align*}
|z(t,x)-z(t,x')| & = \Bigl| x + \int_0^t b(z(s,x)+{\bx}_s, {\bx}_s)ds - x' - \int_0^t b(z(s,x)+{\bx}_s, {\bx}_s)ds\Bigr| \\ 
& \leq |x-x'| + K \int_0^t |z(s,x)-z(s,x')| ds . 
\end{align*}
We conclude by Gronwall's Lemma that 
$$
|z(t,x)-z(t,x')| \leq |x-x'| e^{Kt}, 
$$
which proves the claim. We now discuss the Lipschitz regularity of $\mu_0\to \mathcal Y^{\mu_0}$. Let $\mu_0,\mu_1\in \mathcal P_1(\R^d)$ and let us set $z_i(t,x)= \mathcal Y^{\mu_i}(t,x,\bx)$ (for $i=0,1$). We have 
\begin{align*}
|z_0(t,x)-z_1(t,x)| & =  \Bigl| -\bar \mu_0 +\int_0^t b(z_0(s,x)+{\bx}_s, {\bx}_s)ds - \int_0^t\int_{\R^d} b(z_0(s,y)+{\bx}_s, {\bx}_s)\mu_0(dy)ds\\
& \qquad + \bar \mu_1 - \int_0^t b(z_1(s,x)+{\bx}_s, {\bx}_s)ds + \int_0^t\int_{\R^d} b(z_1(s,y)+{\bx}_s, {\bx}_s)\mu_1(dy)ds \Bigr|  \\ 
& \leq {W}_1(\mu_0,\mu_1) + K \int_0^t |z_0(s,x)-z_1(s,x)| ds\\
& \qquad  +  \int_0^t\int_{\R^d} |b(z_0(s,y)+{\bx}_s, {\bx}_s)-b(z_1(s,y)+{\bx}_s, {\bx}_s)|\mu_0(dy)ds \\
& \qquad +  \int_0^t \Bigl| \int_{\R^d} b(z_1(s,y)+{\bx}_s, {\bx}_s)(\mu_1(dy)-\mu_0(dy))  \Bigr| ds
\end{align*}
As $y \to b(z_1(s,y)+{\bx}_s, {\bx}_s)$ is Lipschitz continuous with a Lipschitz constant $C$ independent of $t\in [0,T]$, $\bx$ and of the $\mu_i$, we infer that 
\begin{align*}
|z_0(t,x)-z_1(t,x)| & \leq (1+Ct) {W}_1(\mu_0,\mu_1) + 2 K \int_0^t |z_0(s,x)-z_1(s,x)| ds .
\end{align*}
We conclude by Gronwall's Lemma that 
$$
|z_0(t,x)-z_1(t,x)|  \leq C'  {W}_1(\mu_0,\mu_1) , 
$$
where the constant $C'$ does not depend on $t\in [0,T]$, $x\in \R^d$ and $\bx \in {\mathcal C}$.

Last, if $(y^1_t,\cdots,y^N_t)_{t\in[0,T]}$ solves \eqref{eq:discdyn}, then
$$\forall t\in[0,T],\;\sum_{i=1}^N|y^i_t-\mathcal Y^{m^N_{\vec{x}_0}}(t,x^i_0, {\bx})|\le 2K\int_0^t\sum_{i=1}^N|y^i_s-\mathcal Y^{m^N_{\vec{x}_0}}(s,x^i_0, {\bx})|ds.$$
By Gronwall's lemma, we conclude that
$$
(y^1_t,\cdots,y^N_t)_{t\in[0,T]}=(\mathcal Y^{m^N_{\vec{x}_0}}(t,x^1_0, {\bx}),\cdots,\mathcal Y^{m^N_{\vec{x}_0}}(t,x^N_0, {\bx}))_{t\in[0,T]}.
$$
\end{proof}

\subsection{Continuity of the solution w.r.t. the initial measure} \label{subsec2.2}

   We recall that $P^{\mu_0}$ denotes the distribution of the weak solution $(W^0, X, X^0)$ of \eqref{eq.barXsigma=0BIS} with initial distribution $\mu_0\in\mathcal P_1(\R^d)$ which is made explicit in Remark \ref{rem:PQ}, and that  $Q^{\mu_0}$ is the distribution of $X^0$. We discuss here the regularity property of the map $\mu_0\to P^{\mu_0}$.

   To compare $P^{\mu_0}$ with $P^{\hat\mu_0}$ for $\mu_0, \hat \mu_0\in \mathcal P_1(\R^d)$, and later quantify the propagation of chaos, we introduce the bounded Lipschitz or Fortet Maurier distance.
\begin{definition}\label{def:dBL}
  Let $({\mathcal Z},d_{\mathcal Z})$ be a metric space. For $k\in\N^*$, we endow ${\mathcal Z}^k$ with the distance $
  \sum_{i=1}^k d_{{\mathcal Z}}(z_i,\hat z_i)$. For $M\in (0,+\infty)$, the bounded Lipschitz or Fortet-Maurier distance $d_{{\rm B}_M{\rm L}}$ on ${\mathcal P}({\mathcal Z}^k)$ is defined by
  $$d_{{\rm B}_M{\rm L}}(\eta,\gamma)=\sup\left\{\Big|\int_{\mathcal Z^k}\phi(z)(\eta-\gamma)(dz)\Big|,\;\phi:{\mathcal Z}^k\to\R\mbox{ bounded by $M/2$ and $1$-Lipschitz}\right\},$$ for $\eta,\gamma\in{\mathcal P}({\mathcal Z}^k)$.
\end{definition}

\begin{prop}\label{prop.continuity} For any $\ep\in (0,\frac 12)$, there exists a constant $C_\ep>0$, depending on $\ep$, $\sigma^0$, $b$ and $T$ only, such that, for any $M>0$ and any $\mu_0, \hat \mu_0\in \mathcal P_1(\R^d)$, 
\be\label{eq:depmuk0}
d_{\rm B_ML}(P^{\mu_0},P^{\hat\mu_0})\leq C_\ep \left(W_1(\mu_0,\hat \mu_0)^{1/2}+ W_1(\mu_0,\hat \mu_0)^{1/2-\ep}\right) (W_1(\mu_0,\hat \mu_0)^{1/2} +M).
\ee  
More generally, if we  denote (with a slight abuse of notation) by $P^{\mu_0}(d\bx|\bx^0)$ the conditional distribution of the $X$ component given that $X^0=\bx ^0$ under $P^{\mu_0}$, then, for all $M>0$ and for all $\mu_0, \hat \mu_0\in \mathcal P_1(\R^d)$,
\begin{align}\label{eq:depmuk}
& d_{{\rm B}_M{\rm L}}\left(Q^{\mu_0}(d\bx^0)\prod_{i=1}^kP^{\mu_0}(d\bx^i|\bx^0),Q^{\hat\mu_0}(d\bx^0)\prod_{i=1}^kP^{\hat\mu_0}(d\bx^i|\bx^0)\right) \notag\\
& \qquad\qquad\qquad\qquad  \leq C_\ep \left(W_1(\mu_0,\hat \mu_0)^{1/2}+ W_1(\mu_0,\hat \mu_0)^{1/2-\ep}\right)  \left(k W_1(\mu_0,\hat \mu_0)^{1/2}+ M\right).
\end{align}
\label{prop:couplmu}
\end{prop}

Inequality \eqref{eq:depmuk} will play a central role in the proof of the propagation of chaos below. We will not need there the joint law of the process including the Brownian motion, which is the reason why we removed it from the estimate (in contrast with inequality \eqref{eq:depmuk}). To show Proposition \ref{prop:couplmu}, we need the following remark on the dependence  with respect to the noise ${\bx}$ of the map $\mathcal Y$ introduced in Lemma \ref{lem.mathcalY}. 

\begin{lemma}\label{lem.estiYY} Fix $\mu_0\in \mathcal P_1(\R^d)$. Let $\bx,\hat \bx\in {\mathcal C}$ and assume that there exists $\tau\in [0,T]$ such that $\bx(t)=\hat \bx(t)$ for any  $t\in [\tau, T]$. Then 
$$
|\mathcal Y^{\mu_0}(t,x, \bx)-\mathcal Y^{\mu_0}(t,x, \hat \bx)|\leq C(t\wedge \tau)\qquad \forall (t,x)\in [0,T]\times \R^d, 
$$
where $C$ only depends on the bound of $b$, its Lipschitz constant in the first variable and $T$. 
\end{lemma}

\begin{proof}[Proof of Lemma \ref{lem.estiYY}] Let $z(t,y):=\mathcal Y^{\mu_0}(t,y, \bx)$ and $\hat z(t,y):=\mathcal Y^{\mu_0}(t,y, \hat \bx)$. Using the fact that $b$ is bounded and Lipschitz with constant $K$ in its first variable,  we have
\begin{align*}
|z(t,x)-\hat z(t,x)| & \leq \int_0^t  |b(z(s,x)+{\bx}_s, {\bx}_s)- b(\hat z(s,x)+\hat{\bx}_s, \hat{\bx}_s)|ds\\
&\qquad + \int_0^t\int_{\R^d} | b(z(s,y)+{\bx}_s, {\bx}_s)-b(\hat z(s,y)+\hat{\bx}_s, \hat{\bx}_s)|\mu_0(dy)ds \\ 
& \leq 4\|b\|_\infty (\tau\wedge t) + K\int_{\tau\wedge t}^t |z(s,x)-\hat z(s,x)|ds \\
& \qquad + K \int_{\tau\wedge t}^t\int_{\R^d} |z(s,y)-\hat z(s,y)|\mu_0(dy)ds. 
\end{align*}
Integrating with respect to $\mu_0$ and using Gronwall's lemma implies that 
$$
 \int_{\R^d} |z(t,y)-\hat z(t,y)|\mu_0(dy)ds \leq C(\tau\wedge t), 
$$
where $C$ depends on $\|b\|_\infty$, $K$ and $T$. Plugging this inequality into the estimate of $|z(t,x)-\hat z(t,x)|$ and using once more Gronwall's lemma gives the result. 
\end{proof}
\begin{proof}[Proof of Proposition \ref{prop:couplmu}] We  again assume that $\sigma^0=1$. Let $\mu_0,\hat \mu_0$ be two initial conditions in $\mathcal P_1(\R^d)$. We are going to compare the solution of \eqref{eq.barXsigma=0BIS} started at $\mu_0$ with the one started at $\hat \mu_0$ by  using a coupling by reflection. We leverage on the construction of a solution to \eqref{eq.barXsigma=0BIS} in the proof of Theorem \ref{thm.existenceNIN}. Let $(\Omega, \mathcal F, \F, {\mathbb Q})$ be a probability space supporting a standard Brownian motion $B$ and $k$ independent $\F_0-$measurable random variables $(X^i_0,\hat X^i_0)_{i=1,\dots, k}$, independent of $B$, where $X^i_0$ is of law $\mu_0$, $\hat X^i_0$ is of law $\hat \mu_0$, and such that
$$
\E_{\mathbb Q}\left[ |X^i_0-\hat X^i_0|\right] = W_1(\mu_0, \hat \mu_0) \qquad \forall i=1, \dots, k.
$$
We set $\bar \mu_0=\E_{\mathbb Q}\left[X_0^i\right]$, $\bar{\hat \mu}_0=\E_{\mathbb Q}\left[\hat X_0^i\right]$. When $\bar{\hat \mu}_0\neq\bar \mu_0$, we let $\hat B$ be the standard Brownian motion  such that $\bar{\hat \mu}_0+ \hat B$ is the reflection of $\bar \mu_0+B$ with respect to the hyperplane passing through $(\bar \mu_0+\bar{\hat \mu}_0)/2$ and orthogonal to $\bar{\hat \mu}_0-\bar \mu_0$ up to the merging stopping time$$
\tau := \inf\{ t\in[0,T], \; \bar{\hat \mu}_0+ \hat B_t=\bar \mu_0+B_t\}\mbox{ with convention }\inf\emptyset=T.
$$We also set $\hat B_t=\bar \mu_0-\bar{\hat \mu}_0+B_t$ for $t\in [\tau,T]$. If $\bar{\hat \mu}_0=\bar \mu_0$, we simply set $\hat B=B$. 
We have, for any $p>\frac 12$, 
\be\label{estitau}
{\mathbb Q}\left[\tau\geq t\right] \leq \frac{2|\bar{\hat \mu}_0-\bar \mu_0|}{(2\pi t)^{1/2}}\wedge 1\;\; \text{if $t\in (0,T]$}, \qquad  \E_{\mathbb Q}\left[ \tau^p\right] \leq C_{p,T} |\bar{\hat \mu}_0-\bar \mu_0|,
\ee
where $C_{p,T}$ depends only on $T$ and $p$.   Indeed, $\tau$ has the law of a one-dimensional Brownian first passage time at height $h:=|\bar{\hat \mu}_0-\bar \mu_0|$ capped by $T$. Thus, recalling that the density of the first passage time is $h/ (2\pi t^3)^{1/2} \exp\{-h^2/(2t)\}{\bf 1}_{t>0}$, the first inequality in \eqref{estitau} is immediate and we have 
\begin{align*}
\E_{\mathbb Q}\left[ \tau^p\right]&  = \int_0^T \frac{h}{(2\pi s^3)^{1/2}} s^p e^{-\frac{h^2}{2s}}ds + T^p \int_T^\infty \frac{h}{(2\pi s^3)^{1/2}} e^{-\frac{h^2}{2s}}ds \\
                                  &\le \frac{h}{\sqrt{2\pi}}\left(\int_0^T
s^{p-\frac 32}ds+ T^p\int_T^{+\infty}s^{-\frac 32}ds\right) \le Ch, 
\end{align*}
where $C$ depends on $p$ and $T$ only. Thus the second inequality in \eqref{estitau} also holds. 

We now recall that, according to \eqref{defbarb},
$$
\bar b^{\mu_0}(t,z,{\bx}) = \int_{\R^d} b(\mathcal Y^{\mu_0}(t, y, {\bx})+z,z) \mu_0(dy), \quad 
\bar b^{\hat{\mu_0}}(t,z,{\bx}) = \int_{\R^d} b(\mathcal Y^{\hat \mu_0}(t, y, {\bx})+z,z) \hat \mu_0(dy). 
$$
As  $\bar b^{\mu_0}$ and $\bar b^{\hat \mu_0}$ are bounded and nonanticipating, we can define the probability  measures ${\mathbb P}^{\mu_0}$ and ${\mathbb P}^{\hat \mu_0}$ equivalent  to ${\mathbb Q}$ by 
$$
\frac{d{\mathbb P}^{\mu_0}}{d{\mathbb Q}} =Z_T:= \exp\left\{ \int_0^T \bar b^{\mu_0}(s, \bar \mu_0+B_s, \bar \mu_0 +B)\cdot dB_s - \frac12 \int_0^T |\bar b^{\mu_0}(s, \bar \mu_0+B_s,\bar \mu_0+ B)|^2ds \right\}
$$
and
$$
\frac{d{\mathbb P}^{\hat \mu_0}}{d{\mathbb Q}}=\hat Z_T: = \exp\left\{ \int_0^T \bar b^{\hat \mu_0}(s, \bar{\hat \mu}_0+\hat B_s, \bar{\hat \mu}_0 +\hat B)\cdot d\hat B_s - \frac12 \int_0^T |\bar b^{\hat \mu_0}(s, \bar{\hat \mu}_0+\hat B_s,\bar{\hat \mu}_0+ \hat B)|^2ds \right\}. 
$$
By Girsanov Theorem,  under ${\mathbb P}^{\mu_0}$ and ${\mathbb P}^{\hat \mu_0}$ respectively,  the processes \small
$$
\left(W^{0}_t = B_t -\int_0^t \bar b^{\mu_0}(s, \bar \mu_0+B_s, \bar \mu_0+B)ds\right)_{t\in [0,T]}\; {\rm and}\; \left(\hat W^{0}_t = \hat B_t -\int_0^t \bar b^{\hat \mu_0}(s, \bar{\hat \mu}_0+\hat B_s, \bar{\hat \mu}_0+\hat B)ds\right)_{t\in [0,T]}
$$ \normalsize
are Brownian motions, independent of the initial conditions $(X_0^i)_{i=1, \dots, k}$ and $(\hat X_0^i)_{i=1, \dots, k}$, which are still distributed according to $\mu_0$ and $\hat \mu_0$ respectively and such that 
\be\label{lkzeljnfdgfk}
\E_{{\mathbb P}^{\mu_0}}\left[ |X_0^i-\hat X_0^i|\right] = \E_{{\mathbb P}^{\hat \mu_0}}\left[ |X_0^i-\hat X_0^i|\right] =W_1(\mu_0, \hat \mu_0)\qquad \forall i=1, \dots, k.
\ee
We have seen in the proof of Theorem \ref{thm.existenceNIN} that, if we set, for $i=1, \dots, k$,   
$$
X^i_t:= \mathcal Y^{\mu_0}(t,  X^i_0, \bar \mu_0+B)+\bar \mu_0+B_t, \qquad X^{0}_t= \bar \mu_0+ B_t
$$
and 
$$ 
\hat X^i_t:= \mathcal Y^{\hat \mu_0}(t, \hat X^i_0, \bar{\hat \mu}_0+\hat B)+ \bar{\hat \mu}_0+\hat B_t, \qquad \hat X^{0}_t=  \bar{\hat \mu}_0+ \hat B_t,
$$
then  $(\Omega,  \mathcal F, \F, {\mathbb P}^{\mu_0}, W^0, X^i, X^0)$ and $(\Omega,  \mathcal F, \F, {\mathbb P}^{\hat\mu_0}, \hat W^0, \hat X^i, \hat X^0)$  are weak solutions   to \eqref{eq.barXsigma=0BIS} with initial distributions $\mu_0$ and $\hat \mu_0$ respectively. As the $(X^i)_{i=1, \dots, k}$ are independent given $X^0$, the $(k+1)-$tuple $(X^0, \dots, X^k)$ has law $Q^{\mu_0}(d\bx^0)\prod_{i=1}^kP^{\mu_0}(d\bx^i|\bx^0)$ under ${\mathbb P}^{\mu_0}$. In the same way, $(\hat X^0, \dots, \hat X^k)$ has law $Q^{\hat \mu_0}(d\bx^0)\prod_{i=1}^kP^{\hat \mu_0}(d\bx^i|\bx^0)$ under ${\mathbb P}^{\hat\mu_0}$. 

Let $f:{\mathcal C}^{k+2}\to \R$ be a $1-$Lipschitz continuous test function bounded by $M/2$. Our aim is to estimate 
\begin{align*}
I&:= \left| \E_{{\mathbb P}^{\mu_0}}\left[ f(W^0,X^1, \dots, X^k,X^0)\right]- \E_{{\mathbb P}^{\hat \mu_0}}\left[ f(\hat W^0,\hat X^1, \dots, \hat X^k,\hat X^0)\right] \right| \\
& \leq \E_{{\mathbb P}^{\mu_0}}\left[ \left| f(W^0,X^1, \dots, X^k,X^0)- f(\hat W^0,\hat X^1, \dots, \hat X^k,\hat X^0)\right| \right]\\
& \qquad \qquad + \left| \E_{{\mathbb P}^{\mu_0}}\left[  f(\hat W^0,\hat X^1, \dots, \hat X^k,\hat X^0) \right]- \E_{{\mathbb P}^{\hat \mu_0}}\left[ f(\hat W^0,\hat X^1, \dots, \hat X^k,\hat X^0)\right] \right|\\
& =: I_1+I_2. 
\end{align*}
Using the Lipschitz regularity of $f$, we have: 
\begin{align*}
I_1 \leq \E_{{\mathbb P}^{\mu_0}}\left[ \|W^0-\hat W^0\|_\infty+\sum_{i=1}^k \|X^i - \hat X^i\|_\infty + \|X^0-\hat X^0\|_\infty \right].
\end{align*}
We now estimate each term separately. 
First, note that, by Lemma \ref{lem.mathcalY} and Lemma \ref{lem.estiYY}, we have, for $t\in [0, T]$,  for $i=1, \dots, k$, 
\begin{align*}
|X^i_t-\hat X^i_t| & \leq |\bar \mu_0+B_t-(\bar {\hat \mu}_0+\hat B_t)|{\bf 1}_{t\leq \tau}+ |\mathcal Y^{\mu_0}(t,  X^i_0, \bar \mu_0+B)-\mathcal Y^{\hat \mu_0}(t,  \hat X^i_0, \bar \mu_0+B)| \\
& \qquad + 
|\mathcal Y^{\hat \mu_0}(t,  \hat X^i_0, \bar \mu_0+B)-\mathcal Y^{\hat \mu_0}(t,  \hat X^i_0, \bar{\hat \mu}_0+\hat B)|\\
& \leq  |\bar \mu_0+B_t-(\bar {\hat \mu}_0+\hat B_t)|{\bf 1}_{t\leq \tau}+ C(W_1(\mu_0, \hat \mu_0)+ |X^i_0-\hat X^i_0|+ t\wedge \tau)
\end{align*}
while 
\begin{align*}
|X^0_t-\hat X^0_t|= |\bar \mu_0+B_t-(\bar {\hat \mu}_0+\hat B_t)|{\bf 1}_{t\leq \tau}. 
\end{align*}
On the other hand, 
\begin{align*}
& |W^{0}_t - \hat W^0_t|  \leq |B_t-\hat B_t|+\int_0^t | \bar b^{\mu_0}(s, \bar \mu_0+B_s, \bar \mu_0+B)- \bar b^{\hat \mu_0}(s, \bar{\hat \mu}_0+\hat B_s, \bar{\hat \mu}_0+\hat B)|ds \\ 
& \leq W_1(\mu_0, \hat \mu_0) +   |\bar \mu_0+B_t-(\bar {\hat \mu}_0+\hat B_t)|{\bf 1}_{t\leq \tau}+ \int_0^t | \bar b^{\mu_0}(s, \bar \mu_0+B_s, \bar \mu_0+B)- \bar b^{\hat \mu_0}(s, \bar{\hat \mu}_0+\hat B_s, \bar{\hat \mu}_0+\hat B)|ds .
\end{align*}
Recalling the definition of $\bar b^{\mu_0}$ and $\bar b^{\hat \mu_0}$, we have, for any $s\in[0,T]$, 
\begin{align*}
& | \bar b^{\hat \mu_0}(s, \bar{\hat \mu}_0+\hat B_s, \bar{\hat \mu}_0 +\hat B) -\bar b^{\mu_0}(s, \bar \mu_0+B_s,\bar \mu_0+ B)|\\
&\qquad \leq  \Bigl| \int_{\R^d} b(\mathcal Y^{\hat\mu_0}(s, y, \bar{\hat \mu}_0 +\hat B)+\bar{\hat \mu}_0+\hat B_s,\bar{\hat \mu}_0+\hat B_s) \hat \mu_0(dy)\\
& \qquad \qquad -\int_{\R^d} b(\mathcal Y^{ \mu_0}(s, y, \bar \mu_0+ B)+\bar \mu_0+B_s,\bar \mu_0+B_s)  \mu_0(dy)\Bigr|\\
& \qquad \leq CW_1(\mu_0,\hat \mu_0)+ \int_{\R^d} |b(\mathcal Y^{\mu_0}(s, y, \bar{\hat \mu}_0 +\hat B)+\bar{\hat \mu}_0+\hat B_s,\bar{\hat \mu}_0+\hat B_s)\\
& \hspace{4cm} -b(\mathcal Y^{ \mu_0}(s, y, \bar \mu_0+ B)+\bar \mu_0+B_s,\bar \mu_0+B_s) | \mu_0(dy),
\end{align*}
where,  for the last inequality, we use the fact that $(\mu_0,y)\mapsto\mathcal Y^{\mu_0}(s,y,\bx)$ is Lipschitz uniformly in $(s,\bx)\in[0,T]\times {\mathcal C}$ according to Lemma \ref{lem.mathcalY} while $b$ is Lipschitz in the first variable. Recalling Lemma \ref{lem.estiYY}, we have  therefore (using again the fact that $b$ is bounded and Lipschitz in the first variable)
\be\label{kaejhlreIIH}
 | \bar b^{\hat \mu_0}(s, \bar{\hat \mu}_0+\hat B_s, \bar{\hat \mu}_0 +\hat B) -\bar b^{\mu_0}(s, \bar \mu_0+B_s,\bar \mu_0+ B)|
\leq C(W_1(\mu_0,\hat \mu_0)+{\bf 1}_{s\leq \tau} + \tau {\bf 1}_{s>\tau}). 
\ee
Thus
\begin{align*}
|W^{0}_t - \hat W^0_t| & \leq C(W_1(\mu_0,\hat \mu_0)+\tau+   |\bar \mu_0+B_t-(\bar {\hat \mu}_0+\hat B_t)|{\bf 1}_{t\leq \tau}). 
\end{align*}
Collecting the above inequalities and using  \eqref{lkzeljnfdgfk},  we obtain for any ${\ep'}>0$: 
\begin{align*}
I_1& \leq Ck \left(W_1(\mu_0, \hat \mu_0) + \E_{{\mathbb P}^{\mu_0}}\left[ \sup_{t\leq \tau} |\bar \mu_0+B_t-(\bar {\hat \mu}_0+\hat B_t)|+  \tau\right]\right) \\ 
& \leq Ck\left(W_1(\mu_0, \hat \mu_0) + \E_{{\mathbb Q}}\left[Z_T\sup_{t\leq \tau} |\bar \mu_0+B_t-(\bar {\hat \mu}_0+\hat B_t)|\right]+\E_{{\mathbb Q}}\Big[Z_T\tau\Big]\right) \\
& \leq Ck\left(W_1(\mu_0, \hat \mu_0) + \E_{{\mathbb Q}}\Bigl[|Z_T|^{(1+{\ep'})/{\ep'}}\right]^{{\ep'}/(1+{\ep'})}\Bigl( \E_{\mathbb Q}\left[\sup_{t\leq \tau}|\bar \mu_0+B_t-(\bar {\hat \mu}_0+\hat B_t)|^{1+{\ep'}}\right]^{1/(1+{\ep'})} \\
&\qquad \qquad\qquad\qquad\qquad\qquad \qquad\qquad\qquad\qquad\qquad +\E_{\mathbb Q}\left[\tau^{1+{\ep'}}\right]^{1/(1+{\ep'})}\Bigr)\Bigr),
\end{align*}
where, by Doob's inequality, 
\begin{align*}
 \E_{\mathbb Q}\left[\sup_{t\leq \tau}|\bar \mu_0+B_t-(\bar {\hat \mu}_0+\hat B_t)|^{1+{\ep'}}\right]  & \leq C_{\ep'} \E_{\mathbb Q}\left[|\bar \mu_0+B_\tau-(\bar {\hat \mu}_0+\hat B_\tau)|^{1+{\ep'}}\right] \\ 
& =  C_{\ep'} \E_{\mathbb Q}\left[|\bar \mu_0+B_T-(\bar {\hat \mu}_0+\hat B_T)|^{1+{\ep'}}{\bf 1}_{\tau =T}\right] . 
\end{align*}
Now note that, by H\"{o}lder's inequality and then \eqref{estitau}, 
\begin{align*}
& \E_{\mathbb Q}\left[|\bar \mu_0+B_T-(\bar {\hat \mu}_0+\hat B_T)|^{1+{\ep'}}{\bf 1}_{\tau = T}\right]  \leq 
C_{\ep'} \E_{\mathbb Q}\left[|\bar \mu_0-\bar {\hat \mu}_0|^{1+{\ep'}}{\bf 1}_{\tau = T}\right]+ C_{\ep'} \E_{\mathbb Q}\left[|B_T-\hat B_T|^{1+{\ep'}}{\bf 1}_{\tau = T}\right] \\ 
& \qquad \leq C_{\ep'} |\bar \mu_0-\bar {\hat \mu}_0|^{1+{\ep'}} +C_{\ep'} \E_{\mathbb Q}\left[|B_T-\hat B_T|^{(1+{\ep'})^2/{\ep'}} \right]^{{\ep'}/(1+{\ep'})}  {\mathbb Q}\left[\tau \geq T\right]^{1/(1+{\ep'})} \\
& \qquad \leq C_{\ep'} \Bigl( |\bar \mu_0-\bar {\hat \mu}_0|^{1+{\ep'}} + |\bar \mu_0-\bar {\hat \mu}_0|^{1/(1+{\ep'})}\Bigr). 
\end{align*}
With these inequalities in hand, we come back to the estimate of $I_1$ and, using the boundedness of $b$ to bound $\E_{{\mathbb Q}}\Bigl[|Z_T|^{(1+{\ep'})/{\ep'}}\Bigr]^{{\ep'}/(1+{\ep'})}$ from above, we conclude that 
$$
I_1\leq C_{\ep'} k( W_1(\mu_0, \hat \mu_0)+ W_1(\mu_0, \hat \mu_0)^{1/(1+{\ep'})^2}). 
$$
In order to estimate $I_2$,  we first use that $f$ is bounded by $M/2$ and then Pinsker's inequality to obtain
\begin{align*}
I_2 \leq M d_{\rm TV}\left({\mathbb P}^{\hat \mu_0},{\mathbb P}^{ \mu_0}\right)\leq M \sqrt{ \frac12 H({\mathbb P}^{ \hat \mu_0}|{\mathbb P}^{ \mu_0})} ,
\end{align*} 
where $ H({\mathbb P}^{ \hat \mu_0}|{\mathbb P}^{ \mu_0})$ is the relative entropy of ${\mathbb P}^{ \hat \mu_0}$ with respect to ${\mathbb P}^{ \mu_0}$. To handle this last term, we have to introduce a few notation. Let $(e^1, \dots, e^d)$ be an orthonormal basis of $\R^d$ with $\bar \mu_0-\bar{\hat \mu_0}=|\bar \mu_0-\bar{\hat \mu_0}|e_1$. Let $\hat W^{0}=(\hat W^{0,j},\dots,\hat W^{0,d})$, $B=(B^1, \dots, B^d)$ and $\hat B=(\hat B^1, \dots, \hat B^d)$ in this basis. We have $dB^i_t= d\hat B^i_t$ for $i\geq 2$ and for $i=1$ and $t\geq \tau$, while $dB^1_t=-d\hat B^1_t$ on $[0,\tau]$. Using the notation  $\bar b^{ \mu_0,j}(s)$ and $\bar b^{\hat \mu_0,j}(s)$ for the $j-$th component in the basis of $\bar b^{\mu_0}(s, \bar \mu_0+B_s, \bar \mu_0 +B)$ and of $\bar b^{\hat \mu_0}(s, \bar{\hat \mu}_0+\hat B_s, \bar{\hat \mu}_0 +\hat B)$ respectively,  we obtain \small
\begin{align*}
 & \frac{d {\mathbb P}^{ \hat \mu_0}}{d {\mathbb P}^{ \mu_0}}  = \frac{\hat Z_T}{Z_T}  
  = \exp\Bigl\{ \int_0^T (\sum_{j=2}^d (\bar b^{\hat \mu_0,j}(s) -\bar b^{ \mu_0,j}(s)) dB^j_s -\frac12 \int_0^T\sum_{j=2}^d (|\bar b^{\hat \mu_0,j}(s)|^2 -|\bar b^{ \mu_0,j}(s)) |^2ds \\
& \qquad - \int_0^\tau    (\bar b^{\hat \mu_0,1}(s) +\bar b^{ \mu_0,1}(s)) dB^1_s+ \int_\tau^T    (\bar b^{\hat \mu_0,1}(s) -\bar b^{ \mu_0,1}(s)) dB^1_s -\frac12 \int_0^T(|\bar b^{\hat \mu_0,1}(s)|^2 -|\bar b^{ \mu_0,1}(s)) |^2ds \Bigr\}\\ 
& =   \exp\Bigl\{ \int_0^T (\sum_{j=2}^d (\bar b^{\hat \mu_0,j}(s) -\bar b^{ \mu_0,j}(s)) d\hat W^{0,j}_s +\frac12 \int_0^T\sum_{j=2}^d (|\bar b^{\hat \mu_0,j}(s)-\bar b^{ \mu_0,j}(s)) |^2ds- \int_0^\tau    (\bar b^{\hat \mu_0,j}(s) +\bar b^{ \mu_0,j}(s)) d\hat W^{0,1}_s \\
& \qquad + \frac12 \int_0^\tau |\bar b^{\hat \mu_0,1}(s) +\bar b^{ \mu_0,1}(s)|^2ds 
+ \int_\tau ^T   (\bar b^{\hat \mu_0,1}(s) -\bar b^{ \mu_0,1}(s)) d\hat W^{0,1}(s) +\frac12 \int_{{\tau}}^T|\bar b^{\hat \mu_0,1}(s)-\bar b^{ \mu_0,1}(s) |^2ds \Bigr\}.
\end{align*} \normalsize
Thus, as $\hat W^0$ is a Brownian motion under ${\mathbb P}^{\hat \mu_0}$,  one  can write that
\begin{align*}
&H({\mathbb P}^{ \hat \mu_0}|{\mathbb P}^{ \mu_0}) =\E_{{\mathbb P}^{\hat \mu_0}} \left[\ln\left(\frac{d{\mathbb P}^{ \hat \mu_0}}{d {\mathbb P}^{ \mu_0}}\right)\right]\\&=  \frac12 \E_{{\mathbb P}^{\hat \mu_0}} \Bigl[ \int_0^T \sum_{j=2}^d |\bar b^{\hat \mu_0,j}(s)-\bar b^{ \mu_0,j}(s) |^2ds+ \int_0^\tau |\bar b^{\hat \mu_0,1}(s) +\bar b^{ \mu_0,1}(s)|^2 ds + \int_\tau^T   | \bar b^{\hat \mu_0,1}(s)-\bar b^{ \mu_0,1}(s) |^2ds\Bigr] .
\end{align*}
Coming back to \eqref{kaejhlreIIH} and using the fact that $b$ is bounded, we conclude as above, using \eqref{estitau}, 
 that 
\begin{align*}
& H({\mathbb P}^{ \hat \mu_0}|{\mathbb P}^{ \mu_0}) \leq C(W_1(\mu_0,\hat \mu_0) + \E_{{\mathbb P}^{\hat \mu_0}} \left[\tau\right]) \leq C_{\ep''}( W_1(\mu_0,\hat \mu_0)+ W_1(\mu_0,\hat \mu_0)^{1/(1+\ep'')^2}),
\end{align*}
for any $\ep''>0$.
Hence $I_2\le MC_{\ep''}( W_1(\mu_0,\hat \mu_0)^{1/2}+ W_1(\mu_0,\hat \mu_0)^{1/(2(1+\ep'')^2)})$ and, choosing $\ep'=(1-\ep)^{-1/2}-1$ and $\ep''=(1-2\ep)^{-1/2}-1$, we conclude that  
$$
I \leq  C_\ep (W_1(\mu_0,\hat \mu_0)^{1/2}+ W_1(\mu_0,\hat \mu_0)^{1/2-\ep}) (k W_1(\mu_0,\hat \mu_0)^{1/2} +M). 
$$
Taking the supremum over $f$ gives \eqref{eq:depmuk0} when $k=1$ and \eqref{eq:depmuk} for $f$ not depending on its first variable. 
\end{proof}

\subsection{Propagation of chaos}  \label{subsec2.3}

Our aim is to investigate two particle systems. The first one, that we call the natural one, is the first which comes in mind when looking at the McKean-Vlasov equation \eqref{eq.barXsigma=0BIS}. However, we obtain only a relatively slow convergence rate for this system. To improve the convergence rate, we introduce another particle system in which the initial condition of the empirical mean of the particles is slightly twisted: we then obtain an optimal convergence rate. To simplify the notation, we suppose  from now on that $\sigma^0=1$. 

\subsubsection{The natural particle system}

We first investigate the limit, as $N\to \infty$, of the particle system 
\begin{equation}\label{eq:part1}
\left\{\begin{array}{l}
\ds    X^{i,N}_t=X^{i}_0+\int_0^tb(X^{i,N}_s,\bar X^N_s)ds+W^0_t,\;t\in[0,T],\;i\in\{1,\cdots,N\}, \\
\ds \mbox{ with }\bar X^N_s=\frac 1N\sum_{j=1}^N X^{j,N}_s,
\end{array}\right.
\end{equation}
where the initial random variables $(X^i_0)_{i\in\{1,\cdots,N\}}$ are i.i.d. according to $\mu_0$ and the process $(X^{1,N},\cdots,X^{N,N})$ is adapted to a filtration with respect to which $W^0$ is a Brownian motion. Let us recall the notation $m^N_{{\vec{x}}_0}=\frac 1N\sum_{i=1}^N\delta_{x^i_0}$ for $\vec{x}_0=(x^1_0,\cdots,x^N_0)\in\{\R^d\}^N$ and first discuss the existence of a solution to \eqref{eq:part1}. 

\begin{prop}\label{prop:exunsyspartnat}Under Assumption \eqref{Hyp.noCN}, there exists a unique weak solution to the particle system \eqref{eq:part1}. More precisely, the distribution of $(\bar X^N,X^{1,N},\cdots,X^{N,N})$ is the image of $\mu_0^{\otimes N}(d{\vec{x}}_0)Q^{m^N_{{\vec{x}}_0}}(d{\bx}^0)$ by the map
  $$\{\R^d\}^N\times {\mathcal C}\ni({\vec{x}}_0,{\bx}^0)\mapsto \left({\bx}^0,{\bx}^0+\mathcal Y^{m^N_{{\vec{x}}_0}}(\cdot,x^1_0,{\bx}^0),\cdots,{\bx}^0+\mathcal Y^{m^N_{{\vec{x}}_0}}(\cdot,x^N_0,{\bx}^0)\right)\in {\mathcal C}^{1+N},$$ 
  where we recall that $Q^{\mu_0}$ denotes the distribution of $X^0$ when $(W^0,X,X^0)$ is a weak solution to \eqref{eq.barXsigma=0BIS} with initial distribution $\mu_0$. 
\end{prop}
\begin{remark} 
 Note that exchangeability holds in the sense that the distribution of \\$(\bar X^N,X^{\sigma(1),N},\cdots,X^{\sigma(N),N})$ does not depend on the permutation $\sigma$ of $\{1,\cdots,N\}$.
\end{remark}

\begin{proof} 
  We choose $(\vec{X}_0,\bar X^N,W^0)$ distributed according to the image of the probability measure $\mu_0^{\otimes N}(d{\vec{x}}_0)Q^{m^N_{{\vec{x}}_0}}(d{\bx}^0)$ by
$({\vec{x}}_0,{\bx}^0)\mapsto ({\vec{x}}_0,{\bx}^0,{\bx}^0-\bar{m^N_{{\vec{x}}_0}}-\int_0^\cdot\bar b^{m^N_{{\vec{x}}_0}}(s,{\bx}^0_s,{\bx}^0)ds)$. Since for each $\mu\in \mathcal P_1(\R^d)$, the image of $Q^{\mu}(d{\bx}^0)$ by ${\bx}^0\mapsto {\bx}^0-\bar{\mu}-\int_0^\cdot\bar b^{\mu}(s,{\bx}^0_s,{\bx}^0)ds)$ is the Wiener measure and does not depend on $\mu$, $W^0$ is a Brownian motion independent from $\vec{X}_0=(X^1_0,\cdots,X^N_0)$ where the components $X^i_0$ are i.i.d. according to $\mu_0$. Moreover, for each $t\in[0,T]$, we keep the independence between $(W^0_s,\bar X^N_s)_{s\in[0,t]}$ and $(W^0_s-W^0_t)_{s\in[t,T]}$. Using \eqref{defbarb}, we also have 
\begin{align*}
   {\bar X}^N_t=\bar{m^N_{{\vec{X}}_0}}+\int_0^t\frac 1N\sum_{i=1}^Nb(\mathcal Y^{m^N_{{\vec{X}}_0}}(s,X^i_0, {{\bar X}^N})+{{\bar X}^N}(s), {{\bar X}^N}(s))ds+W^0_t,
\end{align*}
where, by definition, \small
$$\mathcal Y^{m^N_{{\vec{X}}_0}}(t,x, {\bx}) = x-\bar{m^N_{{\vec{X}}_0}}+\int_0^t  b(\mathcal Y^{m^N_{{\vec{X}}_0}}(s,x, {\bx})+{\bx}_s, {\bx}_s)ds - \int_0^t\frac 1N\sum_{i=1}^Nb(\mathcal Y^{m^N_{{\vec{X}}_0}}(s,X^i_0, {\bx})+{\bx}_s, {\bx}_s)ds. 
$$ \normalsize
We now set $X^{i,N}_t={\bar X}^N_t+\mathcal Y^{m^N_{{\vec{X}}_0}}(t,X^i_0, {\bar X}^N)$. Then $(\bar X^N,X^{1,N},\cdots,X^{N,N})$ has the distribution given in the statement. Moreover,
\begin{align*}
   X^{i,N}_t=X^i_0+\int_0^tb(X^{i,N}_s,{\bar X}^N_s)ds+W^0_t\mbox{ and }{\bar X}^N_t=\frac 1N\sum_{j=1}^NX^{i,N}_t,
\end{align*}
so that $(X^{1,N},\cdots,X^{N,N})$ is a weak solution to \eqref{eq:part1}.

Conversely let us consider a weak solution to the particle system dynamics \eqref{eq:part1} under the probability measure $\P$. Then, setting ${\vec{X}}_0=(X^1_0,\cdots,X^N_0)$, we have
\begin{align*}
   X^{i,N}_t-\bar X^N_t=X^i_0-\bar{m^N_{{\vec{X}}_0}}+\int_0^tb(X^{i,N}_s-\bar X^N_s+\bar X^N_s,\bar X^N_s)-\frac{1}{N}\sum_{j=1}^Nb(X^{j,N}_s-\bar X^N_s+\bar X^N_s,\bar X^N_s)ds
\end{align*}
By the last assertion in Lemma \ref{lem.mathcalY}, this implies that $\left(X^{i,N}_\cdot-\bar X^N_\cdot\right)_{i=1,\cdots,N}=\left(\mathcal Y^{m^N_{{\vec{X}}_0}}(\cdot,X^i_0,\bar X^N)\right)_{i=1,\cdots,N}$. Hence
$$\bar X^N_t=\bar{m^N_{{\vec{X}}_0}}+\int_0^t\frac 1N \sum_{i=1}^N b(\mathcal Y^{m^N_{{\vec{X}}_0}}(s,X^i_0,\bar X^N)+\bar{X}^N_s,\bar X^N_s)ds+W^0_t.$$
By Girsanov's theorem, under $\Q$ such that \begin{align*}
   \frac{d\Q}{d\P}&=e^{-\int_0^T\frac 1N b(\mathcal Y^{m^N_{{\vec{X}}_0}}(s,X^i_0,\bar X^N)+\bar{X}^N_s,\bar X^N_s).dW^0_s-\frac 12\int_0^T|\frac 1N b(\mathcal Y^{m^N_{{\vec{X}}_0}}(s,X^i_0,\bar X^N)+\bar{X}^N_s,\bar X^N_s)|^2ds}\\&=e^{-\int_0^T\frac 1N b(\mathcal Y^{m^N_{{\vec{X}}_0}}(s,X^i_0,\bar X^N)+\bar{X}^N_s,\bar X^N_s).d\bar X^N_s+\frac 12\int_0^T|\frac 1N b(\mathcal Y^{m^N_{{\vec{X}}_0}}(s,X^i_0,\bar X^N)+\bar{X}^N_s,\bar X^N_s)|^2ds},
                                            \end{align*}
                                            $\left(B_t=W^0_t+\int_0^t\frac 1N \sum_{i=1}^N b(\mathcal Y^{m^N_{{\vec{X}}_0}}(s,X^i_0,\bar X^N)+\bar{X}^N_s,\bar X^N_s)ds\right)_{t\in[0,T]}$ is a Brownian motion independent from the initial conditions $X^i_0$ which remain i.i.d. according to $\mu_0$. Therefore for any bounded and measurable map $\phi:{\mathcal C}^{1+N}\to\R$, $\E\left[ \phi(\bar X^N,X^{1,N}, \dots, X^{N,N})\right]$ is equal to
\begin{align*}
\E^\Q\bigg[& \phi\left(\bar{m^N_{{\vec{X}}_0}}
+B, \mathcal Y^{m^N_{{\vec{X}}_0}}(\cdot,X^1_0, \bar{m^N_{{\vec{X}}_0}}
+B)+\bar{m^N_{{\vec{X}}_0}}+B, \dots,  \mathcal Y^{m^N_{{\vec{X}}_0}}(\cdot,X^N_0, \bar{m^N_{{\vec{X}}_0}}
+B)+\bar{m^N_{{\vec{X}}_0}}
+B\right)\\&e^{\int_0^T\frac 1N b(\mathcal Y^{m^N_{{\vec{X}}_0}}(s,X^i_0,\bar{m^N_{{\vec{X}}_0}}
+B)+\bar{m^N_{{\vec{X}}_0}}
+B_s,\bar{m^N_{{\vec{X}}_0}}
+B_s).dB_s-\frac 12\int_0^T|\frac 1N b(\mathcal Y^{m^N_{{\vec{X}}_0}}(s,X^i_0,\bar{m^N_{{\vec{X}}_0}}
+B)+\bar{X}^N_s,\bar X^N_s)|^2ds}\bigg].
\end{align*}
This ensures weak uniqueness.
\end{proof}

We now present the propagation of chaos for the particle system \eqref{eq:part1}. 
\begin{theorem}\label{propchaosbissansbruitindiv}
Let ${\mathcal L}(\bar X^{N},X^{1,N},\cdots,X^{k,N})$ denote the distribution of the empirical mean and $k$ components in the particle system \eqref{eq:part1} and ${\mathcal L}(X^0,X^1,\cdots,X^k)$ the distribution of the $X^0$ component of the weak solution to \eqref{eq.barXsigma=0BIS} given by Theorem \ref{thm.existenceNIN} together with $k$ conditionally independent copies $(X^1,\cdots,X^k)$ distributed according to the conditional law of the $X$ component given $X^0$.  Under Assumption \eqref{Hyp.noCN}, for each $\ep\in(0,1/2)$, there exists a finite constant $C_\ep$ such that, for any $N\ge k\ge 1$, 
   \begin{align*}
  & d_{{\rm B}_M{\rm L}}({\mathcal L}(\bar X^{N},X^{1,N},\cdots,X^{k,N}),{\mathcal L}(X^0,X^1,\cdots,X^k)) \le M\left(1-\frac{N!}{N^k(N-k)!}\right)\\
   & \qquad +C_\ep \int_{\vec{x}_0} \left(W_1(m^N_{\vec{x}_0}, \mu_0)^{1/2}+ W_1(m^N_{\vec{x}_0}, \mu_0)^{1/2-\ep}\right)  \left(k W_1(m^N_{\vec{x}_0\in \{\R^d\}^N}, \mu_0)^{1/2}+ M\right)\mu_0^{\otimes N}(d\vec{x}_0). 
 \end{align*}
 \end{theorem}
\begin{remark} Note that the right-hand side tends to $0$ as $N$ tends to infinity. Indeed, for the first term, we have $1-\frac{N!}{N^k(N-k)!}\le \frac{k(k-1)}{2N}$, since one easily checks by induction on  $k\in\{0,\cdots,N\}$ that $\frac{N!}{(N-k)!}\ge N^k-N^{k-1}\frac{k(k-1)}2$. The second term tends to $0$ by the Glivenko-Cantelli law of large numbers and, following Fournier and Guillin~\cite{FG15}, the convergence rate is actually algebraic in $N$. 
\end{remark}
\begin{proof}
We first note that, for any $N\ge k\ge 1$ and by the triangle inequality, 
\begin{align}\label{zkqjhesdn}
& d_{{\rm B}_M{\rm L}}({\mathcal L}(\bar X^{N},X^{1,N},\cdots,X^{k,N}),{\mathcal L}(X^0,X^1,\cdots,X^k))\notag\\
& \;  \leq d_{{\rm B}_M{\rm L}}\Bigl({\mathcal L}(\bar X^{N},X^{1,N},\cdots,X^{k,N}), \int_{\vec{x}_0}\prod_{i=1}^kP^{m^N_{\vec{x}_0}}(d\bx^i|\bx^0)Q^{m^N_{\vec{x}_0}}(d\bx^0)\mu_0^{\otimes N}(d\vec{x}_0)\Bigr) \\ 
& \qquad  + d_{{\rm B}_M{\rm L}} \bigg(\int_{\vec{x}_0}\prod_{i=1}^kP^{m^N_{\vec{x}_0}}(d\bx^i|\bx^0)Q^{m^N_{\vec{x}_0}}(d\bx^0)\mu_0^{\otimes N}(d\vec{x}_0),\prod_{i=1}^kP^{\mu_0}(d\bx^i|\bx^0)Q^{\mu_0}(d\bx^0)\bigg). \notag
\end{align}    
In order to estimate the first term in the right-hand side of the inequality above, we recall that, according to Proposition \ref{prop:exunsyspartnat},
  \small \begin{align} \label{eq:loisyspartinit} 
    {\mathcal L}&(\bar X^{N},X^{1,N},\cdots,X^{k,N})\notag \\
   & =\int_{\vec{x}_0\in \{\R^d\}^N} \delta_{\left(\bx^0+{\mathcal Y}^{m^N_{\vec{x}_0}}(\cdot,x^1_0,
    \bx^0),\cdots,\bx^0+{\mathcal Y}^{m^N_{\vec{x}_0}}(\cdot,x^k_0,
                                                        \bx^0)\right)}(d\bx^1,\cdots,d\bx^k)\mu_0^{\otimes N}(d\vec{x}_0)Q^{m^N_{\vec{x}_0}}(d\bx^0)\notag\\
 &=\int_{\vec{x}_0\in \{\R^d\}^N,(y^1_0,\cdots,y^k_0)\in \{\R^d\}^k} \delta_{\left(\bx^0+{\mathcal Y}^{m^N_{\vec{x}_0}}(\cdot,y^1_0,
    \bx^0),\cdots,\bx^0+{\mathcal Y}^{m^N_{\vec{x}_0}}(\cdot,y^k_0, \bx^0)\right)}(d\bx^1,\cdots,d\bx^k)\\
 & \qquad \qquad \qquad\qquad\qquad\qquad\qquad \times \delta_{(x^1_0,\cdots,x^k_0)}(dy^1_0,\cdots,dy^k_0)Q^{m^N_{\vec{x}_0}}(d\bx^0)\mu_0^{\otimes N}(d\vec{x}_0).\notag
  \end{align}\normalsize 
  On the other hand, by Remark \ref{rem:PQ}, $P^{m^N_{\vec{x}_0}}(d\bx|\bx^0)=\ds{\int_{y_0\in\R^d}\delta_{\bx^0+{\mathcal Y}^{m^N_{\vec{x}_0}}(\cdot,y_0,
          \bx^0)}(d\bx)m^N_{\vec{x}_0}(dy_0)}$, so that
  \begin{align}\label{eq:loisyspartinit2}
& \int_{\vec{x}_0\in \{\R^d\}^N} \prod_{i=1}^kP^{m^N_{\vec{x}_0}}(d\bx^i|\bx^0)Q^{m^N_{\vec{x}_0}}(d\bx^0)\mu_0^{\otimes N}(d\vec{x}_0) \notag \\
&  \qquad = 
  \int_{\vec{x}_0\in \{\R^d\}^N, (y^1_0,\cdots,y^k_0)\in \{\R^d\}^k}\delta_{\left(\bx^0+{\mathcal Y}^{m^N_{\vec{x}_0}}(\cdot,y^1_0,
    \bx^0),\cdots,\bx^0+{\mathcal Y}^{m^N_{\vec{x}_0}}(\cdot,y^k_0,
          \bx^0)\right)}(d\bx^1,\cdots,d\bx^k)\notag\\
          & \qquad\qquad\qquad \times (m^N_{\vec{x}_0})^{\otimes k}(dy^1_0,\cdots,dy^k_0)Q^{m^N_{\vec{x}_0}}(d\bx^0)\mu_0^{\otimes N}(d\vec{x}_0) . 
       \end{align}
In order to compare the two expressions, we need to replace  $\delta_{(x^1_0,\cdots,x^k_0)}(dy^1_0,\cdots,dy^k_0)$ in \eqref{eq:loisyspartinit} by $(m^N_{\vec{x}_0})^{\otimes k}(dy^1_0,\cdots,dy^k_0)$. This can be done by following the proof of \cite[Proposition 2.2]{Sz06}.
  For $N\ge k\ge 1$, we have
  $$\int_{\vec{x}_0\in \{\R^d\}^N}\delta_{(x^{1}_0,\cdots,x^{k}_0)}\mu_0^{\otimes N}(d\vec{x}_0)=\mu_0^{\otimes k}=\frac{(N-k)!}{N!}\sum_{f\in{\mathcal I}_{k,N}}\int_{\vec{x}_0\in \{\R^d\}^N}\delta_{(x^{f(1)}_0,\cdots,x^{f(k)}_0)}\mu_0^{\otimes N}(d\vec{x}_0)$$ where the summation is over the set ${\mathcal I}_{k,N}$ of one-to-one functions $f:\{1,\cdots,k\}\to\{1,\cdots,N\}$ which is of cardinality $N!/(N-k)!$ while 
  $$\int_{\vec{x}_0}(m^N_{\vec{x}_0})^{\otimes k}\mu_0^{\otimes N}(d\vec{x}_0)=\frac 1{N^k}\sum_{f\in{\mathcal F}_{k,N}}\int_{\vec{x}_0}\delta_{(x^{f(1)}_0,\cdots,x^{f(k)}_0)}\mu_0^{\otimes N}(d\vec{x}_0)$$
  where the summation is over the set ${\mathcal F}_{k,N}$ of functions $f:\{1,\cdots,k\}\to\{1,\cdots,N\}$ which is of cardinality $N^k$. As a consequence,
  \begin{align*}
    1-\frac{N!}{N^k(N-k)!}&\ge d_{\rm TV}\left(\int_{\vec{x}_0}(m^N_{\vec{x}_0})^{\otimes k}\mu_0^{\otimes N}(d\vec{x}_0),\int_{\vec{x}_0}\delta_{(x^{1}_0,\cdots,x^{k}_0)}\mu_0^{\otimes N}(d\vec{x}_0)\right)\\
    &=d_{\rm TV}\left(\int_{\vec{x}_0}(m^N_{\vec{x}_0})^{\otimes k}\otimes Q^{m^N_{\vec{x}_0}}\mu_0^{\otimes N}(d\vec{x}_0),\int_{\vec{x}_0}\delta_{(x^{1}_0,\cdots,x^{k}_0)}\otimes Q^{m^N_{\vec{x}_0}}\mu_0^{\otimes N}(d\vec{x}_0)\right).
  \end{align*}
       We deduce with \eqref{eq:loisyspartinit} and \eqref{eq:loisyspartinit2} that 
  \begin{align*}
   d_{\rm TV}\bigg({\mathcal L}(\bar X^{N},X^{1,N},\cdots,X^{k,N}), \int_{\vec{x}_0}\prod_{i=1}^kP^{m^N_{\vec{x}_0}}(d\bx^i|\bx^0)Q^{m^N_{\vec{x}_0}}(d\bx^0)\mu_0^{\otimes N}(d\vec{x}_0)\bigg)\le 1-\frac{N!}{N^k(N-k)!}.
  \end{align*}
This complete the estimate of the first term in \eqref{zkqjhesdn}, since, as the test functions in $d_{{\rm B}_M{\rm L}}$ are bounded by $M/2$, we have that $d_{{\rm B}_M{\rm L}}\le Md_{\rm TV}$. To complete the proof, we now turn to the estimate of the second term  in \eqref{zkqjhesdn}. We have, by \eqref{eq:depmuk},
\begin{align*}
  d_{{\rm B}_M{\rm L}}& \bigg(\int_{\vec{x}_0}\prod_{i=1}^kP^{m^N_{\vec{x}_0}}(d\bx^i|\bx^0)Q^{m^N_{\vec{x}_0}}(d\bx^0)\mu_0^{\otimes N}(d\vec{x}_0),\prod_{i=1}^kP^{\mu_0}(d\bx^i|\bx^0)Q^{\mu_0}(d\bx^0)\bigg)\\&\le \int_{\vec{x}_0}d_{{\rm B}_M{\rm L}}\left(\prod_{i=1}^kP^{m^N_{\vec{x}_0}}(d\bx^i|\bx^0)Q^{m^N_{\vec{x}_0}}(d\bx^0),\prod_{i=1}^kP^{\mu_0}(d\bx^i|\bx^0)Q^{\mu_0}(d\bx^0)\right)\mu_0^{\otimes N}(d\vec{x}_0)\\
& \le C_\ep \int_{\vec{x}_0} \left(W_1(m^N_{\vec{x}_0}, \mu_0)^{1/2}+ W_1(m^N_{\vec{x}_0}, \mu_0)^{1/2-\ep}\right)  \left(k W_1(m^N_{\vec{x}_0}, \mu_0)^{1/2}+ M\right)\mu_0^{\otimes N}(d\vec{x}_0).
\end{align*}
\end{proof}

\subsubsection{The well-prepared particle system} 

Let us now consider a second particle system for which we recover propagation of chaos with optimal rate of convergence $N^{-1/2}$ :
\begin{align}\label{PartSystNIN}
\begin{cases}
   &\ds X^{i,N}_t=X^i_0+\sigma^0 W^0_t+\int_0^t b(X^{i,N}_s,X^{0,N}_s)ds\mbox{ with }\\
   &\ds X^{0,N}_s= \frac 1N\sum_{j=1}^N X^{j,N}_s+ \bar \mu_0- \frac 1N\sum_{j=1}^NX^j_0.
\end{cases}
\end{align}
As in \eqref{eq:part1}, $W^0$ is a Brownian motion independent of the initial conditions $( X^i_0)_{i=1,\cdots,N}$  i.i.d. according to $\mu_0$. Note that this system is a small variant of the  more natural one \eqref{eq:part1} where the empirical mean $\bar X^{N}_s=\frac 1N\sum_{j=1}^N X^{j,N}_s$ is replaced by $X^{0,N}_s= \bar X^{N}_s+ \bar \mu_0- \frac 1N\sum_{j=1}^NX^j_0$ as the second variable of the drift of each particle. The main advantage of this replacement, which can be seen as a control variates variance reduction technique, is that now the initial position of this second variable is the same in the particle system and the limiting stochastic differential equation \eqref{eq.barXsigma=0BIS}. Since $b$ is merely measurable in its second variable, this will improve the convergence rate. 
Our aim is to prove that \eqref{PartSystNIN} admits a unique weak solution and that it converge in law to the solution of \eqref{eq.barXsigma=0BIS}.
For $\vec{x}_0=(x^1_0,\cdots,x^N_0)\in\{\R^d\}^N$ and ${\bx}\in {\mathcal C}$, we denote by $({\mathcal Y}^{i,N}(t,\vec{x}_0,{\bx}))_{t\in[0,T],1\le i\le N}$ the unique solution to the $N\times d$ dimensional ordinary differential equation
\begin{equation}
\begin{array}{l}
 \ds  {\mathcal Y}^{i,N}(t,\vec{x}_0,{\bx})= x^i_0-\bar \mu_0+\int_0^tb({\mathcal Y}^{i,N}(s,\vec{x}_0,{\bx})+{\bx}_s,{\bx}_s) \\
\ds    \qquad\qquad\qquad \qquad  -\frac 1N\sum_{j=1}^Nb({\mathcal Y}^{j,N}(s,\vec{x}_0,{\bx})+{\bx}_s,{\bx}_s)ds,\;t\in[0,T],\;1\le i\le N,\label{eq:edN}
\end{array}
\end{equation}
 given by the Cauchy-Lipschitz theorem when $b$ is bounded and Lipschitz continuous in its first variable.
We also set 
\begin{equation}
   b_N(t,\vec{x}_0,z,{\bx})=\frac 1 N\sum_{j=1}^Nb({\mathcal Y}^{j,N}(t,\vec{x}_0,{\bx})+z,z)\mbox{ for }(t,z)\in[0,T]\times\R^d,\label{eq:defbN}
\end{equation}to rewrite
\begin{equation}
   {\mathcal Y}^{i,N}(t,\vec{x}_0,{\bx})= x^i_0-\bar \mu_0+\int_0^tb({\mathcal Y}^{i,N}(s,\vec{x}_0,{\bx})+{\bx}_s,{\bx}_s)-b_N(s,\vec{x}_0,{\bx}_s,{\bx})ds.\label{ynmod}
\end{equation}

\begin{proposition} \label{prop:exunPartSystNIN}Under Assumption \eqref{Hyp.noCN}, there exists a weak solution to the particle system \eqref{PartSystNIN}. Moreover, the distribution of any weak solution $(X^{i,N})_{i=0, \dots, N}$ is given by the image 
  of $$\exp\left(\int_0^T b_N(s,\vec{x}_0,{\bx}^0_s,{\bx}^0).d{\bx}^0_s-\frac 12\int_0^T|b_N(s,\vec{x}_0,{\bx}^0_s,{\bx}^0)|^2ds\right)\mu_0^{\otimes N}(d\vec{x}_0){\mathcal W}^{\bar\mu_0}(d{\bx}^0)$$ by
$$\{\R^d\}^N\times {\mathcal C}\ni({\vec{x}}_0,{\bx}^0)\mapsto \left({\bx}^0,{\bx}^0+\mathcal Y^{1,N}(\cdot,\vec{x}_0,{\bx}^0),\cdots,{\bx}^0+\mathcal Y^{N,N}(\cdot,\vec{x}_0,{\bx}^0)\right)\in {\mathcal C}^{1+N}.$$
\end{proposition}
\begin{remark}
   Note that for $\vec{x}_0=(x^1_0,\cdots,x^N_0)\in\{\R^d\}^N$ and ${\bx}\in {\mathcal C}$, $(\mathcal Y^{m^N_{{\vec{x}}_0}}(\cdot,x^i_0,{\bx}^0))_{1\le i\le N}$ and $({\mathcal Y}^{i,N}(t,\vec{x}_0,{\bx}))_{1\le i\le N}$ are closely related : the term $-\bar \mu_0$ in the right-hand side of the ordinary differential equation \eqref{eq:edN} defining the second system is simply replaced by $-\bar{m^N_{{\vec{x}}_0}}=-\frac 1N\sum_{j=1}^Nx^j_0$ to get the ordinary differential equation \eqref{eq:discdyn} defining the first system. 
 \end{remark} \begin{proof}
  Let $(\vec{X}_0,X^{0,N})$ with $\vec{X}_0=(X^1_0,\cdots,X^N_0)$ be distributed according to $$\exp\left(\int_0^t b_N(s,\vec{x}_0,{\bx}^0_s,{\bx}^0).d{\bx}^0_s-\frac 12\int_0^t|b_N(s,\vec{x}_0,{\bx}^0_s,{\bx}^0)|^2ds\right)\mu_0^{\otimes N}(d\vec{x}_0){\mathcal W}^{\bar \mu_0}(d{\bx}^0).$$ By Girsanov's theorem, $(W^0_t=X^{0,N}_t-\bar\mu_0-\int_0^t b_N(s,\vec{X}_0,X^{0,N}_s,X^{0,N})ds)_{t\in[0,T]}$ is a Brownian motion independent from the initial vector $\vec{X}_0$ with components i.i.d. according to $\mu_0$. Setting $X^{i,N}_t=X^{0,N}_t+{\mathcal Y}^{i,N}(t,\vec{X}_0,X^{0,N})$, we get, using \eqref{ynmod} for the first equality and \eqref{eq:defbN} for the third one,
\begin{align*}
   &X^{i,N}_t= X^i_0+W^0_t+\int_0^t b(X^{i,N}_s,X^{0,N}_s)ds\mbox{ with }\\&X^{0,N}_t=
\bar \mu_0+W^0_t+\int_0^t b_N(s,\vec{X}_0,X^{0,N}_s,X^{0,N})ds=\bar \mu_0+W^0_t+\int_0^t\frac 1N\sum_{j=1}^Nb(X^{j,N}_s,X^{0,N}_s)ds\\&\phantom{X^{0,N}_t}=\bar \mu_0+\frac 1N\sum_{j=1}^N(X^{j,N}_t-X^j_0).\end{align*}
Thus, $(X^{i,N})_{i=1, \dots, N}$ is a weak solution of \eqref{PartSystNIN}. 

Let us conversely suppose that $(X^{i,N})_{i=1, \dots, N}$ is a weak solution of \eqref{PartSystNIN}. Then, $X^{0,N}_t=\bar\mu_0+W^0_t+\int_0^t\frac 1N\sum_{j=1}^Nb(X^{j,N}_s,X^{0,N}_s)ds$ and, by uniqueness for \eqref{eq:edN}, $(X^{i,N}-X^{0,N})_{i=1, \dots, N}=({\mathcal Y}^{i,N}(\cdot,\vec{X}_0,X^{0,N}))_{i=1, \dots, N}$ where $\vec{X}_0=(X^1_0,\cdots,X^N_0)$. Hence
\begin{align*}
   X^{0,N}_t&=\bar\mu_0+W^0_t+\int_0^t\frac 1N\sum_{j=1}^Nb({\mathcal Y}^{j,N}(s,\vec{X}_0,X^{0,N})+X^{0,N}_s,X^{0,N}_s)ds\\&=\bar\mu_0+W^0_t+\int_0^tb_N(s,\vec{X}_0,X^{0,N}_s,X^{0,N})ds
\end{align*}

By Girsanov's theorem, we deduce that $(X^{i,N})_{i=0, \dots, N}$ is distributed according to the image of $\exp\left(\int_0^t b_N(s,\vec{x}_0,{\bx}^0_s,{\bx}^0).d{\bx}^0_s-\frac 12\int_0^t|b_N(s,\vec{x}_0,{\bx}^0_s,{\bx}^0)|^2ds\right)\mu_0^{\otimes N}(d\vec{x}_0){\mathcal W}^{\bar \mu_0}(d{\bx_0})$ by $$\{\R^d\}^N\times {\mathcal C}\ni({\vec{x}}_0,{\bx}^0)\mapsto \left({\bx}^0,{\bx}^0+\mathcal Y^{1,N}(\cdot,\vec{x}_0,{\bx}^0),\cdots,{\bx}^0+\mathcal Y^{N,N}(\cdot,\vec{x}_0,{\bx}^0)\right).$$
\end{proof}


We now discuss the propagation of chaos for the well prepared particle system \eqref{PartSystNIN}. 

\begin{theorem}\label{thm.cvparticleNOIN} For $k\in\{1,\cdots,N\}$, let ${\mathcal L}(X^{0,N},X^{1,N},\cdots,X^{k,N})$ denote the distribution of the second variable of the drift coefficient together with $k$ particles in the particle system \eqref{PartSystNIN} and ${\mathcal L}(X^0,X^1,\cdots,X^k)$ the distribution of the $X^0$ component of the weak solution to \eqref{eq.barXsigma=0BIS} given by Theorem \ref{thm.existenceNIN} together with $k$ conditionally independent copies $(X^1,\cdots,X^k)$ distributed according to the conditional law of the $X$ component given $X^0$.
 Under Assumption \eqref{Hyp.noCN}, there is a constant $C>0$ such that, for any $N\ge k\ge 1$, 
 $$
d_{{\rm B}_{k}{\rm L}}({\mathcal L}(X^{0,N},X^{1,N},\cdots,X^{k,N}),{\mathcal L}(X^0,X^1,\cdots,X^k))\le\frac {C k}{\sqrt{N}}.$$
 \end{theorem}
The proof relies on the next lemma the proof of which is postponed.
\begin{lemma}\label{lemvitl2} For initial positions $(X^j_0)_{j\ge 1}$ i.i.d. according to $\mu_0$,
   $$\sup_{N\in{\mathbb N}^*,\ j\in\{1,\cdots,N\},\ {\bx}\in {\mathcal C}}\E\left[\left(\sqrt{N}\sup_{t\in[0,T]}|{\mathcal Y}^{j,N}(t,(X^1_0,\cdots,X^N_0),{\bx})-{\mathcal Y}^{\mu_0}(t,X^j_0,  \bx)|\right)^\rho\right]<\infty$$
   for any $\rho\geq 1$. 
 \end{lemma}
 \begin{remark}In Definition \ref{def:dBL}, we could replace $\sum_{i=1}^k d_{{\mathcal Z}}(z_i,\hat z_i)$ by $\left(\sum_{i=1}^k d_{{\mathcal Z}}(z_i,\hat z_i)^\rho\right)^{1/\rho}$ to measure the Lipschitz continuity on ${\mathcal Z}^k$ leading to some distance denoted by $d_{{\rm B}_M{\rm L}_\rho}$.  For the particle system \eqref{eq:part1}, the choice $\rho=1$ was dictated by the integrability of the initial condition. For the particle system \eqref{PartSystNIN}, the convergence of which no longer relies on the convergence of $\frac 1N\sum_{j=1}^N X^j_0$ to $\bar\mu_0$ for $X^j_0$ i.i.d. according to $\mu_0$, we may consider other values of $\rho$. 
Replacing $1$ by $\rho>1$ enables to improve the dependence on $k$ in the right-hand side of  \eqref{eq:tvklip} from the present $k=k^{1/1}$ to $k^{1/\rho}$ and therefore to check that $$\exists C<\infty,\;\forall N\ge k\ge 1,\;d_{{\rm B}_{k^{1/\rho}}{\rm L}_\rho}({\mathcal L}(X^{0,N},X^{1,N},\cdots,X^{k,N}),{\mathcal L}(X^0,X^1,\cdots,X^k))\le\frac {C k^{1/\rho}}{\sqrt{N}}.$$
Note that this implies Theorem \ref{thm.cvparticleNOIN} since $d_{{\rm B}_k{\rm L}}=d_{{\rm B}_k{\rm L}_1}\le k^{1-\frac 1\rho}d_{{\rm B}_{k^{1/\rho}}{\rm L}_\rho}$, where the inequality is a consequence of 
   $\sum_{i=1}^k d_{{\mathcal Z}}(z_i,\hat z_i)\le k^{1-\frac 1\rho}\left(\sum_{i=1}^k d_{{\mathcal Z}}(z_i,\hat z_i)^\rho\right)^{1/\rho}$.
 \end{remark}
\begin{proof}[Proof of Theorem \ref{thm.cvparticleNOIN}]
Our first goal is to build a suitable coupling between\\ $\mathcal L^{k,N}= {\mathcal L}(X^{0,N},X^{1,N},\cdots,X^{k,N})$ and $\mathcal L^{k,\infty}={\mathcal L}(X^0,X^1,\cdots,X^k)$. 
  Let under the probability measure $\P$, $(\vec{X}_0,X^0)$ with $\vec{X}_0=(X^1_0,\cdots,X^N_0)$ be distributed according to $\mu_0^{\otimes N}(d\vec{x}_0)Q^{\mu_0}(d{\bx}^0)$, where we recall that $Q^{\mu_0}$ denotes the distribution of the $X^0$ component of the weak solution to \eqref{eq.barXsigma=0BIS}  given by Theorem \ref{thm.existenceNIN}. We set $X^i=X^0+{\mathcal Y}^{\mu_0}(\cdot,X^i_0,X_0)$ for $i\in\{1,\cdots,N\}$. 
  According to Remark \ref{rem:PQ}, for each $k\in\{1,\cdots,N\}$, $(X^0,X^1,\cdots,X^k)$ is distributed according to ${\mathcal L}^{k,\infty}$ and 
  $$
  Q^{\mu_0}(d{\bx}^0)=\exp\left\{ 
\int_0^T \bar b^{\mu_0}(s, {\bx}^0_s, {\bx}^0)\cdot d{\bx}^0_s 
- \frac12 \int_0^T |\bar b^{\mu_0}(s, {\bx}^0_s, {\bx}^0)|^2ds 
\right\}{\mathcal W}^{\bar\mu_0}(d{\bx}^0).
$$ In particular, $(W^0_t=X^0_t-\bar{\mu}_0-\int_0^t\bar b^{\mu_0}(s, X^0_s, X^0)ds)_{t\in[0,T]}$ is a Brownian motion independent of $\vec{X}_0$. 
 Setting \begin{align*}
   Z^N_t(\vec{x}_0,{\bx}^0)=\exp\Big(&\int_0^t (b_N(s,\vec{x}_0,{\bx}^0_s,{\bx}^0)-\bar b^{\mu_0}(s,{\bx}^0_s,{\bx}^0)).d{\bx}^0_s\\&+\frac 12\int_0^t\{|\bar b^{\mu_0}(s,{\bx}^0_s,{\bx}^0)|^2-|b_N(s,\vec{x}_0,{\bx}^0_s,{\bx}^0)|^2\}ds\Big),
\end{align*} we have \begin{align*}
   \exp&\left(\int_0^T b_N(s,\vec{x}_0,{\bx}^0_s,{\bx}^0).d{\bx}^0_s-\frac 12\int_0^T|b_N(s,\vec{x}_0,{\bx}^0_s,{\bx}^0)|^2ds\right)\mu_0^{\otimes N}(d\vec{x}_0){\mathcal W}^{\bar \mu_0}(d{\bx}^0)\\&=Z^N_T(\vec{x}_0,{\bx}^0)\mu_0^{\otimes N}(d\vec{x}_0)Q^{\mu_0}(d{\bx}^0).\end{align*}
We also set $X^{i,N}=X^0+{\mathcal Y}^{i,N}(\cdot,\vec{X}_0,X^0)$ for $i\in\{1,\cdots,N\}$, where $\mathcal Y^{i,N}$ is defined in \eqref{ynmod}. According to Proposition \ref{prop:exunPartSystNIN}, under the probability measure with density \begin{align*}\frac{d{\mathbb Q}_N}{d\P}=Z^N_T(\vec{X}_0,X^0)=\exp\Big(&\int_0^t (b_N(s,\vec{X}_0,X^0_s,X^0)-\bar b^{\mu_0}(s,X^0_s,X^0)).dW^0_s\\&-\frac 12\int_0^t|b_N(s,\vec{X}_0,X^0_s,X^0)-\bar b^{\mu_0}(s,X^0_s,X^0)|^2ds\Big),
\end{align*}  
$(X^0,X^{1,N},\cdots,X^{k,N})$ is distributed according to ${\mathcal L}^{k,N}$. 
Let $\phi:{\mathcal C}^{1+k}\to\R$ be bounded by ${k}/2$ and $1$-Lipschitz. We have using the definition and the exchangeability of $(X^{i,N},X^i)_{i=1,\cdots,N}$ and the independence between $X^0$ and the initial conditions $( X^i_0)_{i\ge 1}$ under ${\mathbb P}$ for the second inequality, and then Lemma \ref{lemvitl2} for the third one,
\begin{align}
  \Bigl|\E_{\mathbb Q_N}[\phi(X^0&,X^{1,N},\cdots,X^{k,N})]-\E[\phi(X^0,X^{1},\cdots,X^k)]\Bigr|\notag \\
  \le &\left|\E_{\mathbb Q_N}[\phi(X^0,X^{1,N},\cdots,X^{k,N})]-\E[\phi(X^0,X^{1,N},\cdots,X^{k,N})]\right|+\E\left[\sum_{i=1}^k\sup_{t\in[0,T]}|X^{i,N}_t-X^i_t|\right]\notag\\\le &2{k}d_{\rm TV}({\mathbb Q_N},{\mathbb P})+{k}\E\left[\E\left[\sup_{t\in[0,T]}|{\mathcal Y}^{1,N}(t,{\bx})-{\mathcal Y}^{\mu_0}(t, X^1_0,{\bx})|\right] \bigg|_{{\bx}=X^0}\right]\notag\\
  \le &2{k}d_{\rm TV}({\mathbb Q_N},{\mathbb P})+\frac{C{k}}{\sqrt{N}}.\label{eq:tvklip}
\end{align}
Since $d_{{\rm B}_{{k}}{\rm L}}({\mathcal L}(X^{0,N},X^{1,N},\cdots,X^{k,N}),{\mathcal L}(X^0,X^1,\cdots,X^k))$ is the supremum of the left-hand side over $\phi:{\mathcal C}^{1+k}\to\R$ bounded by ${k}/2$ and $1$-Lipschitz, we conclude the proof by checking that $\sup_{N\in{\mathbb N}^*}\sqrt{N}d_{\rm TV}({\mathbb Q_N},{\mathbb P})<\infty$. We write $Z^N_t$ instead of $Z^N_t(\vec{X}_0,X^0)$ to lighten the notation.
Using the Burkholder-Davis-Gundy, Cauchy-Schwarz and Doob inequalities, we get
\begin{align}
  d_{\rm TV}({\mathbb Q_N},{\mathbb P})&=\E[|Z^N_T-1|]=\E\left[\left|\int_0^T Z^N_t (b_N(t,\vec{X}_0,X^0_t,X^0)-\bar b^{\mu_0}(t,X^0_t,X^0)).dW^0_t\right|\right]\notag\\
                                       &\le 2\E\left[\left(\int_0^T(Z^N_t)^2|b_N(t,\vec{X}_0,X^0_t,X^0)-\bar b^{\mu_0}(t,X^0_t,X^0)|^2dt\right)^{1/2}\right]\notag\\&\le 2\E\left[\sup_{s\in[0,T]}|Z^N_s|\left(\int_0^T|b_N(t,\vec{X}_0,X^0_t,X^0)-\bar b^{\mu_0}(t,X^0_t,X^0)|^2dt\right)^{1/2}\right]\notag\\
  &\le 4\E[(Z^N_T)^2]^{1/2}\E\left[\int_0^T|b_N(t,\vec{X}_0,X^0_t,X^0)-\bar b^{\mu_0}(t,X^0_t,X^0)|^2dt\right]^{1/2}.\label{tvqnp}
\end{align}
We have
\begin{align*}
(Z^{N}_T)^2& \le e^{4\|b\|_\infty^2T}\exp\Bigl(2\int_0^t (b_N(s,\vec{X}_0,X^0_s,X^0)-\bar b^{\mu_0}(s,X^0_s,X^0)).dW^0_s\\
&\qquad\qquad\qquad  -2\int_0^t|b_N(s,\vec{X}_0,X^0_s,X^0)-\bar b^{\mu_0}(s,X^0_s,X^0)|^2ds\Bigr),
\end{align*}
so that $\E[(Z^N_T)^2]^{1/2}\le e^{2\|b\|_\infty^2T}$. On the other hand, for $t\in[0,T]$, since, by \eqref{eq:defbN} and \eqref{defbarb}, $$b_N(t,\vec{X}_0,{\bx}_t,{\bx})-\bar b^{\mu_0}(t,{\bx}_t,{\bx})=\frac 1N\sum_{j=1}^N b({\mathcal Y}^{j,N}(t,\vec{X}_0,{\bx})+{\bx}_t,{\bx}_t)-\E[b({\mathcal Y}^{\mu_0}(t,X^1_0,{\bx})+{\bx}_t,{\bx}_t)],$$ we have, using the Lipschitz continuity of $b$ with constant $K$ in its first variable, \begin{align*}
   \frac 12|b_N(t,\vec{X}_0,{\bx}_t,{\bx})&-\bar b^{\mu_0}(t,{\bx}_t,{\bx})|^2\le \frac{K^2}{N}\sum_{j=1}^N|{\mathcal Y}^{j,N}(t,\vec{X}_0,{\bx})-{\mathcal Y}^{\mu_0}(t,X^j_0,{\bx})|^2\\&+\left|\frac 1N\sum_{j=1}^N \left(b({\mathcal Y}^{\mu_0}(t,X^j_0,{\bx})+{\bx}_t,{\bx}_t)-\E[b({\mathcal Y}^{\mu_0}(t,X^1_0,{\bx})+{\bx}_t,{\bx}_t)]\right)\right|^2
                                                                                                                                                                                                                                                                                                                                                                     .\end{align*}

                                                                                                                                                                                                                                                                             With Lemma \ref{lemvitl2}, we deduce that
                                                                                                                                                                                                                                                                             $$\sup_{N\in{\mathbb N}^*,{\bx}\in {\mathcal C},t\in [0,T]}N\E\left[|b_N(t,\vec{X}_0,{\bx}_t,{\bx})-\bar b^{\mu_0}(t,{\bx}_t,{\bx})|^2\right]<\infty.$$
Since, by the independence of $X^0$ and the initial conditions $( X^i_0)_{i\ge 1}$ under ${\mathbb P}$, 
\begin{align*}
  \E\left[|b_N(t,\vec{X}_0,X^0_t,X^0)-\bar b^{\mu_0}(t,X^0_t,X^0)|^2\right]=\E\left[\E\left[\left|b_N(t,\vec{X}_0,{\bx}_t,{\bx})-\bar b^{\mu_0}(t,{\bx}_t,{\bx})\right|^2\right] \bigg|_{{\bx}=X^0}\right],
\end{align*}
we deduce that $\sup_{N\in{\mathbb N}^*,t\in [0,T]}N\E\left[|b_N(t,\vec{X}_0,X^0_t,X^0)-\bar b^{\mu_0}(t,X^0_t,X^0)|^2\right]<\infty$. 
Plugging this estimation in \eqref{tvqnp}, we conclude that $\sup_{N\in{\mathbb N}^*}\sqrt{N}d_{\rm TV}({\mathbb Q_N},{\mathbb P})<\infty$.   
\end{proof}

\begin{proof}[Proof of Lemma \ref{lemvitl2}]
By
Jensen's inequality, it is enough to deal with the case when $\rho$ is
an even integer.  We have, setting $\vec{X}_0=(X^1_0,\cdots,X^N_0)$,
  \begin{align*}
    {\mathcal Y}^{i,N}(t,\vec{X}_0,{\bx})-{\mathcal Y}^{\mu_0}(t, X^i_0,{\bx})&=\int_0^tb({\mathcal Y}^{i,N}(s,\vec{X}_0,{\bx})+{\bx}_s,{\bx}_s)-b({\mathcal Y}^{\mu_0}(s, X^i_0,{\bx})+{\bx}_s,{\bx}_s)ds\\&+\int_0^t\frac 1N \sum_{j=1}^N\left(b({\mathcal Y}^{\mu_0}(s,X^j_0,{\bx})+{\bx}_s,{\bx}_s)-b({\mathcal Y}^{j,N}(s,\vec{X}_0,{\bx})+{\bx}_s,{\bx}_s)\right)ds\\&+\int_0^t\E[b({\mathcal Y}^{\mu_0}(s,X^1_0,{\bx})+{\bx}_s,{\bx}_s)]-\frac 1N \sum_{j=1}^Nb({\mathcal Y}^{\mu_0}(s,X^j_0,{\bx})+{\bx}_s,{\bx}_s)ds.
  \end{align*}
  Since $b$ is Lipschitz with constant $K$ in its second variable, we deduce that $\sup_{s\le t}|{\mathcal Y}^{i,N}(s,\vec{X}_0,{\bx})-{\mathcal Y}^{\mu_0}(s, X^i_0,{\bx})|$ is smaller than
\begin{align*}
   &K\int_0^t|{\mathcal Y}^{i,N}(s,\vec{X}_0,{\bx})-{\mathcal Y}^{\mu_0}(s, X^i_0,{\bx})|+\frac 1N \sum_{j=1}^N|{\mathcal Y}^{j,N}(s,\vec{X}_0,{\bx})-{\mathcal Y}^{\mu_0}(s,X^j_0,{\bx})|ds\\&+\int_0^t\bigg|\E[b({\mathcal Y}^{\mu_0}(s,X^1_0,{\bx})+{\bx}_s,{\bx}_s)]-\frac 1N \sum_{j=1}^Nb({\mathcal Y}^{\mu_0}(s,X^j_0,{\bx})+{\bx}_s,{\bx}_s)\bigg|ds.
\end{align*}
Hence, by exchangeability,
\begin{align*}
   \E\left[\sup_{s\le t}|{\mathcal Y}^{i,N}(s,\vec{X}_0,{\bx})-{\mathcal Y}^{\mu_0}(s, X^i_0,{\bx})|^\rho\right]&\le 2\times 3^{\rho-1}K^\rho T^{\rho-1}\int_0^t\E[|{\mathcal Y}^{i,N}(s,\vec{X}_0,{\bx})-{\mathcal Y}^{\mu_0}(s, X^i_0,{\bx})|^\rho]ds\\&\hspace{-5cm}+3^{\rho-1}T^{\rho-1}\int_0^t\E\bigg[\bigg|\frac 1N \sum_{j=1}^N\{b({\mathcal Y}^{\mu_0}(s,X^j_0,{\bx})+{\bx}_s,{\bx}_s)-\E[b({\mathcal Y}^{\mu_0}(s,X^1_0,{\bx})+{\bx}_s,{\bx}_s)]\}\bigg|^\rho\bigg]ds.
\end{align*}
The expectation in the second integral in the right-hand side is bounded by $\frac{C\|b\|_\infty^\rho}{N^{\rho/2}}$ since, as  $\rho$ is an even integer, any term in the expansion of the $\rho$-th power of the sum with indices not equal by pairs has zero expectation. We conclude by Gronwall's lemma.
\end{proof}

\section{The problem with individual noise} \label{sec.3}

We now consider the McKean-Vlasov equation with individual and common noises: 
\be\label{eq.barXsigma=0BISCN}
\left\{\begin{array}{l}
\ds X_t =  X_0+ \int_0^t b(X_s, X^0_s) dt + \sigma W_t+ \sigma^0 W^0_t, \qquad t\in [0,T], \\
\ds X^0_t= \E\left[X_t\ |\ W^0\right], \qquad t\in [0,T], \qquad \mathcal L( X_0)= \mu_0. 
\end{array}\right.
\ee
Throughout this section we assume that
\be\label{HypCN}\begin{array}{c}
\mbox{\rm  $\sigma,\sigma^0>0$,  $\mu_0\in \mathcal P_1(\R^d)$ and} \\
\mbox{\rm   $b:[0,T]\times \R^d\times \R^d\to \R^d$ is Borel measurable and bounded. }\\
\mbox{\rm .}
\end{array}
\ee

\begin{definition}[Weak solution] \label{def.weaksolNINCN} A weak solution to \eqref{eq.barXsigma=0BISCN} consists of a tuple $(\Omega, (\mathcal F_t), \F,{\mathbb P}, W, W^0,X^0,X)$, where  $(\Omega, (\mathcal F_t), \F,{\mathbb P})$ is a filtered probability space with a complete filtration supporting $(W, W^0,X^0,X)$, such that  
\begin{itemize}
\item[(i)] Under ${\mathbb P}$, $W$ and $W^0$ are independent $d-$dimensional $(\mathcal F_t)-$Brownian motions,  the processes $X$ and $X^0$ are $(\mathcal F_t)-$adapted with values in $\R^d$, with $X_0$ distributed according to $\mu_0$, 

\item[(ii)] $X_0$, $W$ and $(W^0, X^0)$ are independent, 

\item[(iii)] The state equation holds: 
$$
 X_t = X_0+ \int_0^t b(X_s, X^0_s) ds + \sigma W_t+\sigma^0 W^0_t.
 $$

\item[(iv)]  $W^0$ is adapted to the filtration generated by $X^0$, and  $X^0_t= \E\left[X_t\ |\ X^0\right]$ for any $t\in [0,T]$. 
\end{itemize}
\end{definition}

\subsection{Existence and uniqueness of a solution} \label{subsec3.1}

The main result of this section is the following existence and uniqueness result of the weak solution to \eqref{eq.barXsigma=0BISCN}.

\begin{theorem}\label{thm.existenceCN} Under Assumption \eqref{HypCN}, there exists a weak solution to \eqref{eq.barXsigma=0BISCN}. In addition, this solution is unique in law. 
\end{theorem}

To simplify the notation, we assume throughout the rest of the section  that $\sigma=\sigma^0=1$, the proof for general $\sigma,\sigma_0>0$ being the same.  \\

Let us first start by building the conditional law of $X_t-X^0_t$ given $X^0$. For $R>0$, $t>s>0$ and $p\geq 1$, we denote by $B_R$ the ball of $\R^d$ centered at $0$ and of radius $R$ and by $L^p([s,t]\times B_R)$ the Lebesgue space of $p$ integrable maps on $(s,t)\times B_R$ for the Lebesgue measure.  

\begin{lemma}\label{lem.mu} There exists a Borel measurable map $\mu:[0,T]\times {\mathcal C}\to \mathcal P_2(\R^d)$ such that, for ${\mathcal W}^{\bar\mu_0}-$a.e. ${\bx}\in {\mathcal C}$, $\mu^{{\bx}} \in C([0,T], \mathcal P_1(\R^d))$ is the unique solution in the sense of distributions to 
\be\label{eq.mu}
\partial_t \mu^{{\bx}}_t - \frac{1}{2} \Delta \mu^{{\bx}}_t +{\rm div}\Bigl(\mu^{{\bx}}_t \Bigl( b(y+{\bx}_t, \bx_t)-\int_{\R^d} b(z+{\bx}_t,\bx_t)\mu^{\bx}_t(dz)\Bigr)\Bigr) =0
\ee
with initial condition $\mu^{\bx}_0(dy)= (Id- \bar\mu_0)\sharp \mu_0$.  In addition, the map $\bx \to \mu^{\bx}$ is nonanticipating: 
$$
\mbox{\rm for ${\mathcal W}^{\bar\mu_0}-$a.e. $\bx, \bx' \in \mathcal C$, $\bx=\bx'$ on $[0,t]$ implies $\mu^{\bx}_t=\mu^{\bx'}_t$ on $[0,t]$.}
$$
 Moreover this solution is stable with respect to $b$: namely, if $(b^n)$ is a sequence of Lipschitz maps from $\R^d\times \R^d$ into $\R^d$ such that 
\be\label{defbn}
\sup_n \|b^n\|_\infty < \infty,\qquad b^n(x,y)\to b(x,y) \;\mbox{\rm for $dx\otimes dy$ a.e. $(x,y)\in \R^d\times \R^d$},
\ee
then for ${\mathcal W}^{\bar\mu_0}-$a.e. ${\bx}\in {\mathcal C}$, the solution  $\mu^{n,\bx}\in C([0,T], \mathcal P_1(\R^d))$ in the sense of distributions to 
\be\label{eq.mun}
\partial_t \mu^{n,{\bx}}_t - \frac{1}{2} \Delta \mu^{n,{\bx}}_t +{\rm div}\Bigl(\mu^{n,{\bx}}_t \Bigl( b^n(y+{\bx}_t, \bx_t)-\int_{\R^d} b^n(z+{\bx}_t, \bx_t)\mu^{n,\bx}_t(dz)\Bigr)\Bigr) =0
\ee
with initial condition $\mu^{n,\bx}_0( dy)= (Id- \bar\mu_0)\sharp \mu_0$, is unique,  has a density which is  locally bounded and continuous on $(0,T]\times \R^d$, and converges to $\mu^{\bx}$ in $C([0,T], \mathcal P_1(\R^d))$, and weakly in $L^p([\ep,T]\times B_R)$ for any $\ep,R>0$ and any $p\in (1, \infty)$. Finally, for any $n$, the map $\bx \to \mu^{n,\bx}$ is continuous from $\mathcal C$ to $C([0,T], \mathcal P_1(\R^d))$.
\end{lemma}

\begin{proof} 

  We use the approximation by $b^n$ to build a solution to \eqref{eq.mu} and then show that \eqref{eq.mu} has a unique solution. 
  We first note that, for any $R>0$ and $p\in[1,+\infty)$, we have 
\begin{align*}
&\int_{{\mathcal C}} \int_{B_R} \int_0^T \left| b^n(x +{\bx}_t,{\bx}_t)-b(x +{\bx}_t,{\bx}_t)\right|^p dt dx {\mathcal W}^{\bar \mu_0}(d{\bx}) \\
& \qquad= \int_{\R^d} \int_{B_R}\int_0^T \Gamma_t(z-\bar\mu_0) \left| b^n(x +z,z)-b(x +z,z)\right|^p dt dx dz,
\end{align*}
where $\Gamma$ is the heat kernel. The last expression tends to $0$ by dominated convergence. Using a diagonal argument, we can find a subsequence, again denoted by $b^n$ and a subset $\tilde{\mathcal C}$ of ${\mathcal C}$ with ${\mathcal W}^{\bar\mu_0}(\tilde{\mathcal C})=1$, such that, for any ${\bx}\in \tilde{\mathcal C}$, $(b^n(\cdot+{\bx}(\cdot),{\bx}(\cdot)))$ converges to $b(\cdot+{\bx}(\cdot),{\bx}(\cdot))$ in $L^p_{loc}([0,T]\times \R^d)$ for any $p\in[1,+\infty)$. 

Let us fix such a ${\bx}\in \tilde{\mathcal C}$. As the $b^n$ are  Lipschitz and bounded, there exists a unique solution $\mu^{n,\bx}$ to \eqref{eq.mun} in the sense of distributions (see \cite{Szn}). Our aim is to show that the family $(\mu^{n,\bx})$ is compact in $C([0,T], \mathcal P_1(\R^d))$ and that any cluster point $\mu^{\bx}$ is a weak solution to \eqref{eq.mu}. Then we will check that the weak solution to \eqref{eq.mu} is unique, so that the whole sequence $(\mu^{n,\bx})$ converges to $\mu^{\bx}$. This in turn implies that the map $\bx \to \mu^{\bx}$ is measurable from $\mathcal C$ to $C([0,T], \mathcal P_1(\R^d))$ and nonanticipating, once this has been checked for the $\mu^{n,\bx}$.

We now check that the family $(\mu^{n,\bx})$  is compact in $C([0,T], \mathcal P_1(\R^d))$. Still following  \cite{Szn}, we know that, for $t\in[0,T]$, $\mu^{n,\bx}_t$ is the law of $Z_t$, where $(Z_t)_{t\in[0,T]}$ is the solution to the McKean-Vlasov SDE
\be\label{leakzjrsndtfgv}
Z_t= Z_0+\int_0^t (b^n(Z_s+\bx_s, \bx_s) - \E\left[b^n(Z_s+\bx_s,\bx_s)\right])ds +W_t,
\ee
where $Z_0$ has law $(Id- \bar\mu_0)\sharp \mu_0$ and is independent of the $d$-dimensional Brownian motion $W$. The map $\bx \to \mu^{n, \bx}$ is nonanticipating by uniqueness of the solution and continuous from $\mathcal C$ to $C([0,T], \mathcal P_1(\R^d))$ by classical stability of the solution of a McKean-Vlasov equation with Lipschitz coefficients. Moreover,
   $$\sup_{t\in[0,T]}|Z_t|\le |Z_0|+\sup_{t\in[0,T]}|W_t|+2\sup_n\|b^n\|_\infty T.$$ As a consequence, for $R>6\sup_n\|b^n\|_\infty T$, ${\bf 1}_{\{\sup_{t\in[0,T]}|Z_t|\ge R\}}\le {\bf 1}_{\{|Z_0|\ge R/3\}}+{\bf 1}_{\{\sup_{t\in[0,T]}|W_t|\ge R/3\}}$ and
   \begin{align*}
      \E\bigg[\sup_{t\in[0,T]}|Z_t|{\bf 1}_{\{\sup_{t\in[0,T]}|Z_t|\ge R\}}\bigg]\le& \E\bigg[|Z_0|{\bf 1}_{\{|Z_0|\ge R/3\}}\bigg]+\E\bigg[\sup_{t\in[0,T]}|W_t|{\bf 1}_{\{\sup_{t\in[0,T]}|W_t|\ge R/3\}}\bigg]\\&+\Big(\E\Big[\sup_{t\in[0,T]}|W_t|\Big]+2\sup_n\|b^n\|_\infty T\Big){\mathbb P}(|Z_0|\ge R/3)\\&+\Big(\E[|Z_0|]+2\sup_n\|b^n\|_\infty T\Big){\mathbb P}\bigg(\sup_{t\in[0,T]}|W_t|\ge R/3\bigg),
   \end{align*}
where the right-hand side does not depend on $\bx$ and goes to $0$ as $R\to\infty$. We deduce that 
 $$
 \lim_{R\to\infty} \sup_{\bx \in \mathcal C}\sup_n \sup_{t\in [0,T]} \int_{|x|\ge R}|x|\mu^{n,\bx}_t(dx) = 0. 
 $$
Thus the $\mu^{n,\bx}_t$ takes values in a compact subset of $\mathcal P_1(\R^d)$.  Moreover, for $0\le s\le t\le T$, \begin{align*}
   W_1(\mu^{n,\bx}_t,\mu^{n,\bx}_s)\le \E[|Z_t-Z_s|]&\le 2\sup_n\|b^n\|_\infty(t-s)+\E[|W_t-W_s|]\\&=2\sup_n\|b^n\|_\infty(t-s)+\frac{\Gamma((d+1)/2)}{\Gamma(d/2)}\sqrt{2(t-s)},
\end{align*}
so that the mappings $[0,T]\ni t\mapsto\mu^{n,\bx}_t\in{\mathcal P}_1(\R^d)$ are uniformly continuous. By the Ascoli-Arzela theorem, we deduce that $(\mu^{n,\bx})$ is compact in $C([0,T], \mathcal P_1(\R^d))$. In order to be able to take the limit $n\to\infty$ in the nonlinear term in \eqref{eq.mun}, we remark that, in addition, for $t>0$, $\mu^{n,\bx}_t$ has a density $m^{n}(t,x)$ with respect to the Lebesgue measure which satisfies $dx$-a.e. the equality
 \begin{align}
m^n(t,x) & = \Gamma_t\ast \mu_0(x+\bar \mu_0)- \int_0^t D\Gamma_{t-s}\ast \Bigl(m^n(s,\cdot) \times \notag\\
& \qquad \qquad \qquad  \Bigl( b^n(\cdot+\bx_s, \bx_s)- \int_{\R^d} b^n(z+\bx_s, \bx_s)m^n(s,z)dz \Bigr)\Bigr)(x)ds ,\label{eq:mild}
\end{align}
where $\Gamma$ is the heat kernel. This can be proved by using Itô's formula to compute $\E[\phi(t,Z_t)]$ where $\phi(s,x)=\Gamma_{t-s}*\psi(x)$ for some function $\psi:\R^d\to\R$ twice continuously differentiable and bounded together with its first and second order derivatives. Denoting by $\|\cdot\|_{L^p}$ the $L^p$ norm with respect to the Lebesgue measure on $\R^d$, we have $\|\Gamma_t*\mu_0(\cdot-\bar\mu_0)\|_{L^p}\le \|\Gamma_t\|_{L^p}\le Ct^{-\frac{d(p-1)}{2p}}$ and $\|D\Gamma_{t}\|_{L^p}\le Ct^{-\frac 12-\frac{d(p-1)}{2p}}$ so that, by Young's inequality, for $p\in[1,2d]$,$$\left\|D\Gamma_{t-s}*m^n(s,\cdot) \Bigl( b^n(\cdot+\bx_s, \bx_s)- \int_{\R^d} b^n(z+\bx_s, \bx_s)m^n(s,z)dz \Bigr)\right\|_{L^{\frac{2dp}{2d-p}}}\le C(t-s)^{-\frac 34}\|m^n(s,\cdot)\|_{L^p},$$
where the constant $C$ does not depend on $n$ since the $b^n$ are uniformly bounded. With \eqref{eq:mild}, we deduce by induction on $k$ that for $k\in\{0,\cdots,2d\}$, $\sup_{n}\sup_{t\in (0,T]}t^{k/4}\|m^n(t,\cdot)\|_{L^{\frac{2d}{2d-k}}}<\infty$. In particular for $k=2d$, $\sup_{n}\sup_{t\in (0,T]}t^{d/2}\|m^n(t,\cdot)\|_{\infty}<\infty$.

We infer that, up to a subsequence again denoted in the same way, $(\mu^{n,\bx}_t)_{t\in[0,T]}$ converges to some $(\tilde \mu_t)_{t\in[0,T]}$ in $C([0,T], \mathcal P_1(\R^d))$ and $(m^{n}(\cdot, \cdot))$ converges weakly to $(\tilde m(\cdot, \cdot))$ in $L^p([\ep,T]\times B_R)$ for any $R,\ep>0$ and $p\in[1,+\infty)$. Moreover, $dt$-a.e., $\tilde m(t,\cdot)$ is a density of $\tilde \mu_t$ with respect to the Lebesgue measure. As $m^{n}$ solves  \eqref{eq.mun}, it satisfies, for any smooth test function $\phi=\phi(t,y)$ with a compact support in $[0,T)\times \R^d$, 
\begin{align}\label{ksfdhjlgkn}
 & \int_{\R^d} \phi(0,y-\bar\mu_0)\mu_0(dy) +\int_0^T\int_{\R^d} \Bigl(\partial_t\phi(t,x)+\frac12\Delta_y \phi(t,y) \\
 &+ D_y\phi(t,y)\cdot \Bigl(b^n(y+{\bx}_t,{\bx}_t) -\int_{\R^d} b^n(z+{\bx}_t,{\bx}_t)m^{n}(t,z) dz \Bigr)\Bigr) m^{n}(t,y)dydt =0 \notag.
 \end{align}
By the weak convergence of the $m^n$ (say, in $L^2_{loc}([0,T]\times \R^d)$ and the strong convergence (again in $L^2_{loc}$) of $b^n(\cdot+ \bx(\cdot), \cdot)$, we have, for any $\ep>0$ small,  
\begin{align*}
&\lim_{n\to\infty}\int_\ep^T\int_{\R^d} \Bigl(\partial_t\phi(t,x)+ \frac12\Delta_y \phi(t,y) + D_y\phi(t,y)\cdot b^n(y+{\bx}_t,{\bx}_t)\Bigr) m^{n}(t,z) dz dt\\
& \qquad =\int_\ep^T\int_{\R^d} \Bigl(\partial_t\phi(t,x)+ \frac12\Delta_y \phi(t,y) + D_y\phi(t,y)\cdot b(y+{\bx}_t,{\bx}_t) \Bigr) \tilde m(t,z) dzdt.
 \end{align*}
 On the other hand, for any $R>0$ large, 
 \begin{align*}
& \Bigl|  \int_\ep^T\int_{\R^d}  D_y\phi(t,y)\cdot (\int_{B_R} b^n(z+{\bx}_t,{\bx}_t)m^{n}(t,z) dz)  m^{n}(t,y)dydt \\
& \qquad - \int_\ep^T\int_{\R^d}  D_y\phi(t,y)\cdot (\int_{B_R} b^n(z+{\bx}_t,{\bx}_t)m^{n}(t,z) dz)  \tilde m (t,y)dydt\Bigr|\\
& \leq   \int_\ep^T\Bigl| \int_{B_R} b^n(z+{\bx}_t,{\bx}_t)m^{n}(t,z) dz\Bigr|  \Bigl| \int_{\R^d}  D_y\phi(t,y) (m^{n}(t,y)-\tilde m (t,y))dy \Bigr| dt \\
& \leq \|b^n\|_\infty   \|D^2_{yy}\phi\|_\infty \int_\ep^T W_1(\mu^{n,\bx}_t,\tilde \mu_t)dt . 
\end{align*}
Thus 
\begin{align*}
&\lim_{n\to\infty}\bigg|\int_\ep^T\int_{\R^d}  D_y\phi(t,y)\cdot \left(\int_{B_R} b^n(z+{\bx}_t,{\bx}_t)m^{n}(t,z) dz\right)  m^{n}(t,y)dydt \\
&\qquad -\int_\ep^T\int_{B_R} b^n(z+{\bx}_t,{\bx}_t)   \cdot \Bigl( \int_{\R^d}  D_y\phi(t,y) \tilde m(t,y)dy\Bigr) m^n(t,z) dz)  dt\bigg|=0 ,
\end{align*}
where 
$$
(t,z)\to b^n(z+{\bx}_t,{\bx}_t)   \cdot \Bigl( \int_{\R^d}  D_y\phi(t,y) \tilde m(t,y)dy\Bigr)
$$
strongly converges  in $L^2([\ep, T]\times B_R)$  to 
$$
(t,z)\to b(z+{\bx}_t,{\bx}_t)   \cdot \Bigl( \int_{\R^d}  D_y\phi(t,y) \tilde m(t,y)dy\Bigr), 
$$
while $m^n$ weakly converges to $\tilde m$ in $L^2([\ep, T]\times B_R)$. This shows that 
\begin{align*}
& \lim_{n\to\infty} \int_\ep^T\int_{\R^d}  D_y\phi(t,y)\cdot \left(\int_{B_R} b^n(z+{\bx}_t,{\bx}_t)m^{n}(t,z) dz\right)  m^{n}(t,y)dydt\\
& \qquad = \int_\ep^T\int_{\R^d}  D_y\phi(t,y)\cdot \left(\int_{B_R} b(z+{\bx}_t,{\bx}_t)\tilde m(t,z) dz\right) \tilde m(t,y)dydt. 
\end{align*}
The small time contribution
\begin{align*}
 & \int_0^\ep\int_{\R^d} \Bigl( \partial_t\phi(t,y) +\frac12\Delta_y \phi(t,y) + D_y\phi(t,y)\cdot \Bigl(b^n(y+{\bx}_t,{\bx}_t) -\int_{\R^d} b^n(z+{\bx}_t,{\bx}_t)m^{n}(t,z) dz \Bigr)\Bigr) m^{n}(t,y)dydt
 \end{align*}
 is bounded by $C\ep$ thanks to the bound on $\|b^n\|_\infty$, while the term 
 \begin{align*}
 & \int_0^T\int_{\R^d} D_y\phi(t,y)\cdot \Bigl(\int_{\R^d\backslash B_R} b^n(z+{\bx}_t,{\bx}_t)m^{n}(t,z) dz \Bigr)m^{n}(t,y)dydt 
 \end{align*}
 is bounded by $CR^{-1}$ thanks to the uniform control on the first moment of the $\mu^{n,\bx}_t$ and Markov's inequality. So we can pass to the limit in \eqref{ksfdhjlgkn} to obtain 
 \begin{align*}
 & \int_{\R^d} \phi(0,y)\mu_0(dy) +\int_0^T\int_{\R^d} \Bigl(\partial_t\phi(t,y) + \frac12\Delta_y \phi(t,y) \\
 &\quad \quad\quad\quad \quad\quad\quad\quad\quad + D_y\phi(t,y)\cdot \Bigl(b(y+{\bx}_t,{\bx}_t) -\int_{\R^d} b(z+{\bx}_t,{\bx}_t)\tilde m(t,z) dz \Bigr)\Bigr)\tilde m(t,y)dydt =0 .
 \end{align*}
 Therefore $\tilde m$ is a weak solution to the equation \eqref{eq.mu}.  
 
 We now check the uniqueness of the solution to \eqref{eq.mu}, which will prove the convergence of the whole sequence $(\mu^{n,\bx})$. Let $\mu^{\bx}, \tilde \mu^{\bx}$ be two solutions to \eqref{eq.mu} for a fixed $\bx\in \tilde{\mathcal C}$. The weak formulation extends to smooth test functions $\phi$ that are bounded together with their first and second order derivatives. For $t\in (0,T]$, by choosing $\phi(s,x)=\Gamma_{t-s}*\psi(x)$ where $\psi:\R^d\to\R$ is smooth and bounded together with its first and second order derivatives, we obtain that  $\mu^{\bx}
_t$ and $\tilde \mu^{\bx}_t$ admit densities $m(t,x)$ and $\tilde m(t,x)$ with respect to the Lebesgue measure which solve mild equations like \eqref{eq:mild} with $b^n$ replaced by $b$. Setting $\nu= m-\tilde m$, we deduce that
 \begin{align*}
 \nu(t,x)& = - \int_0^t D\Gamma_{t-s}\ast \Bigl(\nu(s,\cdot) \Bigl(b(\cdot+\bx_s, \bx_s)- \int_{\R^d} b(z+\bx_s, \bx_s)m(s,z)dz)\Bigr) \\
 & \qquad\qquad\qquad\qquad - \tilde m(s,\cdot)\int_{\R^d} b(z+\bx_s, \bx_s)\nu(s,z)dz\Bigr)(x)ds ,
 \end{align*}
 where $\Gamma$ is the heat kernel.
 Thus, we have
 \begin{align*}
\|\nu(t,\cdot)\|_{L^1} & \leq \int_0^t \frac{C}{(t-s)^{1/2}} \Bigl\| \nu(s,\cdot) \Bigl( b(\cdot+\bx_s, \bx_s) - \int_{\R^d} b(z+\bx_s, \bx_s)m(s,z)dz)\Bigr) \\
& \qquad\qquad\qquad\qquad\qquad - \tilde m(s,\cdot)\int_{\R^d} b(z+\bx_s, \bx_s)\nu(s,z)dz\Bigr\|_{L^1}ds  \\
 & \leq \int_0^t \frac{C}{(t-s)^{1/2}}  \|b\|_\infty \|\nu(s,\cdot)\|_{L^1} ds.
 \end{align*}
Iterating this inequality, we get $\|\nu(t,\cdot)\|_{L^1}\le C^2\|b\|_\infty^2\pi\int_0^t \|\nu(s,\cdot)\|_{L^1} ds$. By Gronwall's lemma,  we  derive that $\nu(t,\cdot)=0$ for any $t\in (0,T]$. 
 \end{proof}

 We set, for 
${\bx}\in {\mathcal C}$,  
\be\label{defbarbCN}
\bar b(t,z,{\bx}) = \int_{\R^d} b(y+z, z) \mu^{\bx}_t(dy),\quad(t,z)\in[0,T]\times\R^d
\ee
where $\mu^{\bx}$ is given by Lemma \ref{lem.mu}. Note that since $\mu^{\bx}$ is nonanticipating according to this Lemma, so is $\bar b$ : $\bar b(t,\cdot,{\bx})$ depends on $\bx$ only through its restriction to $[0,t]$.
Let ${\mathcal C}^2$ be endowed with the topology of uniform convergence and with the associated Borel $\sigma-$algebra $\bar{\mathcal F}$.  We denote by $({\bx}^0_t,\by_t)_{t\in[0,T]}$ the canonical process on ${\mathcal C}^2$.
\begin{lemma}\label{lem.mathcalYCN} There exists a unique probability measure $Q(d\bx^0,d\by) = Q^{\bx^0}(d\by) {\mathcal W}^{\bar\mu_0}(d\bx^0)$ on $({\mathcal C}^2, \bar{\mathcal F})$ (with $Q^{\bx^0}$ denoting the conditional law of $\by$ given $\bx^0$ under $Q$),  such that 
\begin{itemize}
\item[(i)] under $Q(d\bx^0,d\by) 
$, the process
$$
\Bigl( {\bx}^0_t\; , \; \by_t-\by_0 - \int_0^t  \left(b(\by_s+{\bx}^0_s,{\bx}^0_s)ds - \bar b(s,\bx^0_s,\bx^0)\right)ds \Bigr)_{t\in[0,T]}, 
$$
 is a $2d-$dimensional Brownian motion starting from $(\bar\mu_0,0)$ for the canonical filtration generated by $(\bx^0,\by)$, and $\by_0$ is distributed according to $(Id-\bar \mu_0)\sharp \mu_0$, 
\item[(ii)] ${\mathcal W}^{\bar \mu_0}(d\bx^0)\otimes dt$-a.e., 
  $\by_t\sharp Q^{\bx^0}= \mu^{\bx^0}_t$. 
\end{itemize}
\end{lemma}

 \begin{proof} 
   Let $W,W^0$ be two independent Brownian motions on a probability space $(\Omega, \mathcal F, \P)$ and $X_0$ be an initial condition of law $\mu_0$ independent of $W,W^0$. We consider a regularization $b^n$ of $b$ as in \eqref{defbn} and, for any $\bx\in {\mathcal C}$, let $Y^{n,\bx}$ be the solution to the McKean-Vlasov equation
$$
Y^{n,\bx}_t = X_0-\bar\mu_0 + \int_0^t b^n(Y^{n,\bx}_s +\bx_s,\bx_s)ds - \int_0^t \E[b^n(Y^{n,\bx}_s+\bx_s,\bx_s)]+W_t.
$$ We denote the distribution of $Y^{n,\bx}_t$ by  $\mu^{n,\bx}_t$.
Let us set $X^0=\bar\mu_0+W^0$ and $\bar Y^n_t = Y^{n,X^0}_t$. Let $Q^n$ be the law of $(X^0,\bar Y^n)$ on ${\mathcal C}^2$. We recall from Lemma \ref{lem.mu} that, as $n\to\infty$,  $\mu^{n,X^0}$ converges for a.s. to $\mu^{X^0}$ in $C([0,T], \mathcal P_{1}(\R^d))$. 

As the $b^n$ are uniformly bounded, classical arguments show that the family $(Q^n)$ is tight. Let us denote again by $(Q^n)$ a weakly converging subsequence and let $Q$ be its limit. The next step consists in showing that, for any smooth map $\phi:(\R^d)^2\to \R$ with a compact support,  the process
\begin{align}\label{Qmartingale}
& M^\phi_t:= \phi(\bx^0_t,\by_t)-\int_0^t \Bigl(\frac{1}{2} \Delta \phi (\bx^0_s,\by_s) +D_y\phi(\bx^0_s,\by_s)\cdot \Bigl(b(\by_s+\bx^0_s,\bx^0_s) 
-\bar b(s,\bx^0_s,\bx^0)\Bigr)\Bigr)ds
\end{align} 
 is a $Q-$martingale. As this statement holds for $Q^n$ and $b^n$ instead of $Q$ and $b$, we just need to check that, for any $k\in \N$, any smooth maps $\phi:\{\R^d\}^2\to \R$ and $\psi:\{\R^d\}^{2k}\to \R$ with a compact support and any $0\le t_1\le  t_2\le \dots\le t_k\le r\le t\le T$,  
\begin{align}\label{jkehsnrdf}
&\lim_{n\to\infty} \E_{Q^n}\Bigl[ \Bigl(\phi(\bx^0_t,\by_t)-\phi(\bx^0_r,\by_r)-\int_r^t \Bigl(\frac12\Delta \phi (\bx^0_s,\by_s) +D_y\phi(\bx^0_s,\by_s)\cdot \Bigl(b^n(\by_s+\bx^0_s,\bx^0_s) \notag\\
&\qquad\qquad\qquad\qquad\qquad\qquad   -\int_{\R^d} b^n(z+\bx^0_s,\bx^0_s)\mu^{n,\bx^0}_s(dz)\Bigr)\Bigr)ds\ \psi(\bx^0_{t_1},\by_{t_1},\dots, \bx^0_{t_k},\by_{t_k})\Bigr] \notag \\
& = \E_{Q}\Bigl[ \Bigl(\phi(\bx^0_t,\by_t)-\phi(\bx^0_r,\by_r)-\int_r^t \Bigl(\frac12\Delta \phi (\bx^0_s,\by_s) +D_y\phi(\bx^0_s,\by_s)\cdot \Bigl(b(\by_s+\bx^0_s,\bx^0_s) \\
&\qquad\qquad\qquad\qquad \qquad\qquad    -\int_{\R^d} b(z+\bx^0_s,\bx^0_s)\mu^{\bx^0}_s(dz))\Bigr)\Bigr)ds \  \psi(\bx^0_{t_1},\by_{t_1},\dots, \bx^0_{t_k},\by_{t_k})\Bigr] \notag . 
\end{align}
To prove the claim \eqref{jkehsnrdf}, we first note that the convergence of all the terms but the one multiplied  by $D_y\phi(\bx^0_s,\by_s)$ is clear. We now focus on this term. Let us set, for $p\leq m\leq n$,  
\begin{align*}
&I_{n,p,m} = \E_{Q^n}\Bigl[ \int_r^t \Bigl(D_y\phi(\bx^0_s,\by_s)\cdot (b^p(\by_s+\bx^0_s,\bx^0_s) \\
&\qquad\qquad\qquad  -\int_{\R^d} b^p(z+\bx^0_s,\bx^0_s)\mu^{m,\bx^0}_s(dz)\Bigr)ds  \,\psi(\bx^0_{t_1},\by_{t_1},\dots, \bx^0_{t_k},\by_{t_k})\Bigr] .
\end{align*}
We denote by $I_{\infty,p,m}$ the corresponding quantity with $Q$ instead of $Q^n$ and use in the same way the notations $I_{\infty,p, \infty}$ and $I_{\infty,\infty,\infty}$. 
Then 
\begin{align*}
|I_{n,n,n}-I_{n,p, n}| &\leq \|\psi\|_\infty\|D\phi\|_\infty \E_{Q^n}\Bigl[ \int_r^t (|b^p(\by_s+\bx^0_s,\bx^0_s)-b^n(\by_s+\bx^0_s,\bx^0_s)| \\
 & \qquad\qquad + \int_{\R^d} |b^p(z+\bx^0_s,\bx^0_s)-b^n(z+\bx^0_s,\bx^0_s)| \mu^{n,\bx^0}_s(dz) )ds \Bigr]\\
 & = 2\|\psi\|_\infty\|D\phi\|_\infty \E_{\P}\left[  \int_r^t  \int_{\R^d} |b^p(z+X^0_s,X^0_s)-b^n(z+X^0_s,X^0_s)| \mu^{n,X^0}_s(dz) ds \right]
 & \leq \ep 
\end{align*}
for a fixed  $p$ large enough (independently of $n\geq p$) thanks to arguments similar to those in the proof of Lemma \ref{lem.mu}. On the other hand, once $p$ is fixed, we have in view of the convergence of $\mu^n$ to $\mu$:
\begin{align*}
|I_{n,p,n}-I_{n,p,m}| &\leq \|\psi\|_\infty\|D\phi\|_\infty \E_{Q^n}\Bigl[ \int_r^t \Bigl|  \int_{\R^d} b^p(z+\bx^0_s,\bx^0_s)(\mu^{m,\bx^0}_s(dz)-\mu^{n,\bx^0}_s(dz)\Bigr| ds \Bigr]\\
&\leq \|\psi\|_\infty\|D\phi\|_\infty {\rm Lip}(b^p) \E_{Q^n}\Bigl[ \int_r^t  {W}_1(\mu^{m,\bx^0}_s,\mu^{n,\bx^0}_s)  ds \Bigr]\\
& =  \|\psi\|_\infty\|D\phi\|_\infty {\rm Lip}(b^p) \E_\P\Bigl[ \int_r^t  {W}_1(\mu^{m,X^0}_s,\mu^{n,X^0}_s)  ds \Bigr] \\
 & \leq \ep, 
\end{align*}
for $m$ large enough, independently of $n\geq m$. 
 In a similar way, we can also choose $p$ and then $m$ large enough such that 
\begin{align*}
|I_{\infty,\infty,\infty}-I_{\infty,p,m}|\le|I_{\infty,\infty,\infty}-I_{\infty,p,\infty}|+|I_{\infty,p,\infty}-I_{\infty,p,m}|  &\leq  \eps. 
\end{align*}
In view of the continuity of $b^p$ (in both variables) and of the continuity of the map $\bx \to \mu^{m,\bx}$ from $\mathcal C$ into $C([0,T], \mathcal P_1(\R^d))$,  we can now choose $n$ large enough such that 
 \begin{align*}
 |I_{n,p,m}-I_{\infty,p,m}| \leq \ep. 
 \end{align*} 
 Combining the above estimates proves our claim \eqref{jkehsnrdf} and thus that the process $M^\phi$ defined in \eqref{Qmartingale} is a $Q-$martingale. 
 This shows
 that $Q$ is a solution of the martingale problem associated with the path-dependent operator
 \begin{align*}
 {\mathcal A}'_t\phi(\bx^0,\by)& = \frac12\Delta u (\bx^0_t,\by_t) +D_yu(\bx^0_t,\by_t)\cdot \Bigl(b(\by_t+\bx^0_t,\bx^0_t) -\int_{\R^d} b(z+\bx^0_t,\bx^0_t)\mu^{\bx^0}_t(dz)\Bigr).
 \end{align*}
 in the sense of \cite[Section 5.4]{KaSh}. As $b$ is bounded, \cite[Proposition 5.4.11]{KaSh} implies that $Q$ is actually a solution to the local martingale problem. Then \cite[Proposition 5.4.6]{KaSh} and its proof show that, under $Q$,  the $2d-$dimensional process 
$$
B_t:= \bigl(\bx^0_t-\bar\mu_0\; ,\;  \by_t-  \by_0 - \int_0^t  \left(b(\by_s+\bx^0_s, \bx^0_s)-\bar b(s,\bx^0_s, \bx^0)\right)ds \bigr)
$$
is a continuous local martingale in the canonical filtration generated by $(\bx,\by)$ with quadratic variation 
$$
\langle B^{(i)}, B^{(j)}\rangle_t = {\bf 1}_{\{i=j\}} t\qquad \forall i,j\in \{1, \dots, 2d\}.
$$
Thus, by Lévy's theorem, $B$ is a $(2d)-$dimensional Brownian motion for the canonical filtration. We rewrite $B$ as $(W^0,W)$, where $W^0$ and $W$ are two independent $d-$dimensional Brownian motions under $Q$.\\

It remains to check that that the conditional law under $Q$ of $\by_t$ given $\bx^0$ is $\mu^{\bx^0}_t$: 
indeed we have, for any $t_1<t_2<\dots< t_k$ and any functions $\phi:\R^d\to \R$ and $\psi:\{\R^d\}^k\to \R$, smooth with a compact support,  
 \begin{align*}
\E_{Q^n}\left[ \phi(\by_t)\psi(\bx^0_{t_1}, \dots, \bx^0_{t_k})\right]  & = \E_{\P}\left[ \phi(Y^{n,X^0}_t)\psi(X^0_{t_1}, \dots, X^0_{t_k})\right]  \\ 
&= \E_{\P}\left[ \int_{\R^d} \phi(z) \mu^{n,X^0}_t(dz) \psi(X^0_{t_1}, \dots, X^0_{t_k})\right] ,
 \end{align*}
 where the first equality holds by the definition of $Q^n$ and the second one  by the definition of $\mu^{n,X^0}$. Taking the limit $n\to\infty$ in this equality and using that $\mu^{n,X^0}_t$ tends a.s. to $\mu^{X^0}_t$ in $\mathcal P_1(\R^d)$, we find that 
 \begin{align*}
  \E_{Q}\left[ \phi(\by_t)\psi(\bx^0_{t_1}, \dots, \bx^0_{t_k})\right] &= \E_{\P}\left[ \int_{\R^d} \phi(z) \mu^{X^0}_t( dz) \psi(X^0_{t_1}, \dots, X^0_{t_k})\right] \\
&= \E_Q\left[ \int_{\R^d} \phi(z) \mu^{\bx^0}_t(dz) \psi(\bx^0_{t_1}, \dots, \bx^0_{t_k})\right].  
 \end{align*}
 This proves that the conditional law under $Q$ of $\by_t$ given $\bx^0$ is $\mu^{\bx^0}_t$ for $Q-$a.e. $\bx^0$. \\
 
Let us finally show that the measure $Q$ is unique. Let $Q^{\bx^0}$ be the conditional law of $\by$ given $\bx^0$: $Q(d\bx^0,d\by)= Q^{\bx^0}(d\by){\mathcal W}^{\bar\mu_0}(d\bx^0)$. Then for ${\mathcal W}^{\bar\mu_0}-$a.e. $\bx^0$, $\by$ is under $Q^{\bx^0}$ a weak solution to the McKean-Vlasov equation: 
 $$
 \by_t= \by_0+ \int_0^t  b(\by_s+\bx^0_s,\bx^0_s)ds - \int_0^t \int_{\R^d} b(z+\bx^0_s,\bx^0_s) \mu^{\bx^0}_s(dz)+W_t, 
 $$
 where $\mu^{\bx^0}_s(dz)$ is the law of $\by_s$ (given $\bx^0$) and $W$ is a Brownian motion independent of $\bx^0$. One can check, like in the proof of Lemma \ref{lem.mu}, that $\by$ is unique in law (see also \cite{MiVe20} for the fact that weak solutions are actually strong and unique), which implies that $Q^{\bx^0}$ is uniquely defined. As a consequence, $Q$ is uniquely defined. 
\end{proof}

\begin{proof}[Proof of Theorem \ref{thm.existenceCN}, existence]
  Let on $(\Omega, \F,{\mathbb Q})$, $(X^0_t,Y_t)_{t\in[0,T]}$ be distributed according to the probability measure $Q$ given by Lemma \ref{lem.mathcalYCN}. For $t\in[0,T]$, we set \begin{align*}
   W_t&=Y_t-Y_0-\int_0^t\left(b(Y_s+X^0_s,X^0_s)-\bar b(s,X^0_s,X^0)\right)ds,\\W^0_t&=X^0_t-\bar\mu_0-\int_0^t \bar b(s,X^0_s, X^0)ds,
  \end{align*} and ${\mathcal F}_t=\sigma((X^0_s,Y_s)_{s\in[0,t]})$. Then, under ${\mathbb Q}$,
  $\left(X^0_t-\bar\mu_0,W_t\right)_{t\in[0,T]}$
  is a $2d$-dimensional ${\mathcal F}_t$-Brownian motion independent from $\bar\mu_0+Y_0$ which is distributed according to $\mu_0$. Moreover, the conditional law of $Y_t$ given $X^0$ is ${\mathbb Q}\otimes dt$ a.e. equal to $\mu^{X^0}_t$. By Girsanov's theorem,  
   under the probability measure ${\mathbb P}$ such that $$
\frac{d{\mathbb P}}{d{\mathbb Q}} = \exp\left\{ \int_0^T \bar b(s,X^0_s, X^0)\cdot dX^0_s - \frac12 \int_0^T |\bar b(s,X^0_s, X^0)|^2ds \right\},
$$
$\left(W^0_t,W_t\right)_{t\in[0,T]}$ is a $2d$-dimensional ${\mathcal F}_t$-Brownian motion. Since $\bar b$ is nonanticipating, $W^0$ is adapted to the filtration generated by $X^0$. Since $\frac{d{\mathbb P}}{d{\mathbb Q}}$ is a function of $X^0$ only, under ${\mathbb P}$,  $\bar\mu_0+Y_0$ is  distributed according to $\mu_0$ and independent from $(W,W^0,X^0)$. Moreover, the conditional law of $Y_t$ given $X^0$ under ${\mathbb P}$ remains ${\mathbb P}\otimes dt$ a.e. equal to $\mu^{X^0}_t$. Also using the definition \eqref{defbarbCN} of $\bar b$, we deduce that
\begin{align*}
 \E_{\P}\left[ Y_t\ |\ X^0\right]  & = \E_{\P}[Y_0\ |\ X^0]  +\int_0^t \left(\E_{\P}\left[ b(Y_s+X^0_s,X^0_s)\ |\ X^0\right]-\bar b(s,X^0_s,X^0)\right)ds+ \E_{\P}[W_t\ |\ X^0]   \\
&=\int_{\R^d} (y-\bar\mu_0)\mu_0(dy)+\int_0^t \left(\int_{\R^d}b(y+X^0_s,X^0_s)\mu^{X^0}_s(dy)-\bar b(s,X^0_s,X^0)\right)ds+0=0.
\end{align*}
   Let us finally set 
$$
X_t=Y_t+X^0_t,\;t\in[0,T].
$$ 
We have
$$X^0_t =\E_{\P}\left[ X^0_t+Y_t\ |\ X^0\right]=\E_{\P}\left[ X_t\ |\ X^0\right].$$
To check that the tuple $(\Omega, (\mathcal F_t), \F,{\mathbb P}, W, W^0,X^0,X)$ is a weak solution to  \eqref{eq.barXsigma=0BISCN} in the sense of Definition \ref{def.weaksolNINCN}, it only remains to prove that the state equation holds. This follows from the above definitions of $W$ and $W^0$:
\begin{align*}
X_t &= Y_t+X^0_t=Y_0+\int_0^t \left(b(X_s,X^0_s)-\bar b(s,X^0_s,X^0)\right)ds+W_t+\bar\mu_0+\int_0^t\bar b(s,X^0_s,X^0)ds+W^0_t\\
&=X_0  +\int_0^t b(X_s,X^0_s)ds+ W_t+ W^0_t.
\end{align*}
\end{proof}

\begin{proof}[Proof of Theorem \ref{thm.existenceCN}, uniqueness] Let $(\Omega, (\mathcal F_t), \F,{\mathbb P}, W, W^0,X^0,X)$ be a weak solution to \eqref{eq.barXsigma=0BISCN}. We set as usual 
  $Y_t= X_t-X^0_t$. We note that, by Definition \ref{def.weaksolNINCN} (iii)-(iv), $Y$ solves 
\be\label{qskjhdkf}
 Y_t = X_0-\bar \mu_0 + \int_0^t b(Y_s+X^0_s, X^0_s) ds -\int_0^t \E\left[b(Y_s+X^0_s, X^0_s)\ |\ X^0\right] ds+ W_t.
 \ee
Let $m^{\bx}_t$ be the conditional law of $Y_t$ given $ X^0=\bx$. Then, as $X_0$, $W$ and $(W^0, X^0)$ are independent and $Y$ satisfies \eqref{qskjhdkf}, $m^{X^0}$ is ${\mathbb P}-$a.s. a solution to the McKean-Vlasov equation \eqref{eq.mu}. Thus $m^{ X^0}_t= \mu^{ X^0}_t$ for any $t$ and $\P-$a.s.,  where $\mu^{\bx}$ is the unique solution to \eqref{eq.mu} built in Lemma \ref{lem.mu} (note that, under $\P$, the law of $X^0$ is equivalent to ${\mathcal W}^{\bar \mu_0}$ by Girsanov's theorem).  

Using that $X^0=X-Y$, we deduce that $X^0$  satisfies
\begin{align*}
X^0_t & = \bar \mu_0 +\int_0^t \E\left[ b(Y_s+X^0_s,X^0_s)\ |\ X^0\right] ds + W^0_t \\
&= \bar \mu_0 +\int_0^t \int_{\R^d} b( y+X^0_s,X^0_s)  \mu^{X^0}_s(dy)ds + W^0_t \\ 
& = \bar \mu_0 +\int_0^t  \bar b(s, X^0_s,X^0)ds +W^0_t, 
\end{align*}
where $\bar b$ is defined by \eqref{defbarbCN}.
By Girsanov's theorem, under $\tilde \P$ such that 
\begin{align*}
\frac{d\tilde \P}{d \P} 
& =\exp\left\{ -\int_0^T \bar b(s, X^0_s, X^0)\cdot dX^0_s + \frac12 \int_0^T |\bar b(s, X^0_s, X^0)|^2ds \right\},
\end{align*}
 $X^0-\bar\mu_0$ is a $d$-dimensional Brownian motion. 
 Since $\frac{d\tilde \P}{d \P}$ is function of $X^0$ only, $X^0-\bar \mu_0$ keeps independent of $X_0$ and of $W$ under $\tilde \P$. Moreover, $\tilde \P$ a.s., for each $t\in[0,T]$, the conditional law of $Y_t$ given $X^0$ is $\mu^{X^0}_t$. We deduce that, under $\tilde\P$, $(X^0,Y)$ is distributed according to the probability measure $Q$ introduced in Lemma \ref{lem.mathcalYCN}. Moreover,
 \begin{align*}
   W^0_t&=X^0_t-\bar\mu_0-\int_0^t \bar b(s,X^0_s,X^0)ds\\
   W_t&=Y_t-Y_0+\int_0^t \left(\bar b(s,X^0_s,X^0)-b(Y_s+X^0_s,X^0_s)\right)ds,
 \end{align*}
 so that $(W^0,W)=\Psi(X^0,Y)$ where $\Psi:{\mathcal C}^2\to{\mathcal C}^2$ is defined by $$\Psi(\bx^0,\by)=\left(\bx^0-\bar\mu_0-\int_0^\cdot \bar b(s,\bx^0_s,\bx^0)ds,\by-\by_0+\int_0^\cdot \left(\bar b(s,\bx^0_s,\bx^0)-b(\by_s+\bx^0_s,\bx^0_s)\right)ds\right).$$
Thus for any bounded measurable map $\phi:{\mathcal C}^4\to \R$, we have 
\begin{align*}
&\E_{\P}\left[ \phi(W^0,W,X,X^0)\right] \\
& = \E_{\tilde \P}\Bigl[ \phi(\Psi(X^0,Y),X^0+ Y,X^0)\exp\Bigl\{ \int_0^T \bar b(s, X^0_s, X^0)\cdot dX^0_s- \frac12 \int_0^T |\bar b(s, X^0_s, X^0)|^2ds \Bigr\} \Bigr].
\end{align*}
We conclude that the distribution of $(W,W^0,X,X^0)$ under $\P$ is the image of $\exp\Bigl\{ \int_0^T \bar b(s, \bx^0_s, \bx^0)\cdot d\bx^0_s- \frac12 \int_0^T |\bar b(s, \bx^0_s, \bx^0)|^2ds \Bigr\} Q(d\bx^0,d\by)$ by
$${\mathcal C}^2\ni(\bx^0,\by)\mapsto (\Psi(\bx^0,\by),\bx^0+\by,\bx^0)\in{\mathcal C}^4.$$

\end{proof}

\subsection{Propagation of chaos}\label{subsec3.2}

Let us consider the following particle system: for $N\ge 1$,   
\begin{align}\label{particleCN}
\begin{cases}
\ds X^{i,N}_t=  X^i_0+ \int_0^t b\Bigl(X^{i,N}_s,X^{0,N}_s\Bigr)ds + \sigma W^i_t+\sigma^0 W^0_t,\;i\in\{1, \dots, N\} \\
\ds X^{0,N}_s =  \bar \mu_0+\int_0^s \frac{1}{N} \sum_{j=1}^N b(X^{j,N}_r, X^{0,N}_r)dr+\sigma^0 W^0_s
\end{cases}
\end{align}
where the initial conditions $ X^i_0$ are i.i.d. of law $\mu_0$ and independent of the $W^i$ and $W^0$ which are independent Brownian motions. Following \cite[Theorem~1]{Ver81}, this equation has a unique weak solution and this solution is strong. The goal of this section is to show the convergence in law of the $(X^{i,N})$ to the solution $X$ of \eqref{eq.barXsigma=0BISCN}. 

\begin{theorem}\label{thm.cvparticleWIN} Under assumption \eqref{HypCN}, there is a constant $C>0$, depending on $\|b\|_\infty$, $T$ and $d$ only, such that 
$$
\forall N\ge k\ge 1,\;d_{\rm TV}({\mathcal L}(X^{0,N},X^{1,N},\cdots,X^{k,N}),{\mathcal L}(X^0,X^1,\cdots,X^{k})) \le C\sqrt{\frac{k+1}N},
$$
 where ${\mathcal L}(X^{0,N},X^{1,N},\cdots,X^{k,N})$ denotes the distribution of  $(X^{0,N},X^{1,N},\cdots,X^{k,N})$ in the particle system \eqref{particleCN} and ${\mathcal L}(X^0,X^1,\cdots,X^k)$ the distribution of the $X^0$ component of the weak solution to \eqref{eq.barXsigma=0BISCN} given by Theorem \ref{thm.existenceCN} together with $k$ conditionally independent copies $(X^1,\cdots,X^k)$ distributed according to the conditional law of the $X$ component given $X^0$.\end{theorem}

\begin{remarks}
\begin{enumerate}
\item \label{rem.propag} Note that in the particle system \eqref{particleCN}, like in \eqref{PartSystNIN} without individual noise, the coupling term 
\begin{align*}
X^{0,N}_t & =  \bar \mu_0+\int_0^t \frac{1}{N} \sum_{j=1}^N b(X^{j,N}_r, X^{0,N}_r)dr+\sigma^0 W^0_t \\
& =\frac{1}{N} \sum_{j=1}^N X^{j,N}_t+ \bar \mu_0-\frac{1}{N} \sum_{j=1}^N X^{j}_0- \frac{\sigma}{N} \sum_{j=1}^N W^i_t
\end{align*}
slightly differs from the more natural mean $\frac{1}{N} \sum_{j=1}^N X^{j,N}_t$. Despite the fact that the terms  $ \bar \mu_0-\frac{1}{N} \sum_{j=1}^N X^{j}_0$ and $\frac{1}{N} \sum_{j=1}^N W^i_t$ converge to $0$ as $N\to\infty$, we were not able to show the propagation of chaos when replacing $X^{0,N}$ by the more natural empirical mean $\frac{1}{N} \sum_{j=1}^N X^{j,N}_t$. It seems possible to deal with the contribution $ \bar \mu_0-\frac{1}{N} \sum_{j=1}^N X^{j}_0$ by first proving regularity of the law of the solution to \eqref{eq.barXsigma=0BISCN} in the initial distribution $\mu_0$ like we did without individual noise. But we do not know how to treat the remaining contribution $\frac{1}{N} \sum_{j=1}^N W^i_t$.

\item The proof of the convergence follows closely the construction of Jabir in \cite{Jab19}. Similar ideas, based on change of measures argument and the Girsanov transform, can be found in Jabin-Wang \cite{JaWa16, JaWa18} and Lacker \cite{La18}. 

\item The result actually gives another proof for the uniqueness in law of the nonlinear process $P(d\bx^0,d\bx)$. 
\end{enumerate}
\end{remarks}

\begin{proof} Throughout the proof we assume for notational simplicity that $\sigma=\sigma^0=1$. Moreover $C$ denotes a constant which depends on dimension only. Note that ${\mathcal L}(X^0,X^1,\cdots,X^k)=P^k$ where $P^k(d\bx^0,\cdots,d\bx^k)= R(d\bx^0)\prod_{i=1}^kP(d\bx^i|\bx^0)$ with $P(d\bx^0,\bx)= R(d\bx^0)P(d\bx|\bx^0)$ denoting the law of the nonlinear process \eqref{eq.barXsigma=0BISCN}.
 Let, under the probability measure ${\mathbb P}$, $(X^0,X^1,\cdots,X^N)$ be distributed according to $P^N(d\bx^0,\cdots,d\bx^N)= R(d\bx^0)\prod_{i=1}^NP(d\bx^i|\bx^0)$. 
  There exist independent $d$-dimensional Brownian motions $(W^0,\cdots,W^N)$ independent from the initial random variables $( X^i_0)_{1\le i\le N}$ i.i.d. according to $\mu_0$ such that
\begin{align*}
  &X^{i}_t= X^i_0+\int_0^tb(X^{i}_s,X^0_s)ds+W^i_t+W^0_t,\; t\in[0,T],\;i\in\{1,\cdots,N\},\\
  &\mbox{ with }X^0_t=\bar\mu_0+\int_0^t\E_{{\mathbb P}}[b(X^{1}_s,X^0_s)|X^0]ds+W^0_t
    ,\; t\in[0,T].
 \end{align*}
Let$$
\Delta^N_s = \E_{{\mathbb P}}[b(X^{1}_s,X^0_s)|X^0]-\frac 1N\sum_{j=1}^Nb(X^{j}_s,X^0_s),
$$
$\ep^i =-1$ if $i=0$ and $\ep^i =1$ otherwise, and 
\begin{align*}
   Z^N_t=\exp\bigg(&\int_0^t\Delta^N_s\cdot \Bigl( \sum_{i=0}^N \ep ^idW^i_s\Bigr)-\frac{N+1}2 \int_0^t\big|\Delta^N_s\big|^2ds\bigg).
\end{align*}
By Girsanov's theorem, under the probability measure ${\mathbb Q}$ with density $\frac{d{\mathbb Q}}{d{\mathbb P}}=Z^N_T$ with respect to ${\mathbb P}$, the random variables $( X^i_0)_{1\le i\le N}$ remain i.i.d. with law $\mu_0$ and independent from the independent Brownian motions $(B^0,B^1,\cdots ,B^N)$ where
\begin{align*}
  &B^0_t=W^0_t+\int_0^t\E_{{\mathbb P}}[b(X^{1}_s,X^0_s)|X^0]-\frac 1N\sum_{j=1}^Nb(X^{j}_s,X^0_s)ds,\; t\in[0,T],\\
  &B^i_t=W^i_t+\int_0^t\frac 1N\sum_{j=1}^Nb(X^{j}_s,X^0_s)-\E_{{\mathbb P}}[b(X^{1}_s,X^0_s)|X^0]ds,\; t\in[0,T],\;i\in\{1,\cdots,N\}.
\end{align*}
Note that for $i\in\{1,\cdots,N\}$ and $t\in[0,T]$, $W^i_t+W^0_t=B^i_t+B^0_t$. Moreover,
\begin{align*}
  &X^{i}_t= X^i_0+\int_0^tb(X^{i}_s,X^0_s)ds+B^i_t+B^0_t,\; t\in[0,T],\;i\in\{1,\cdots,N\},\\
  &\mbox{ with }X^0_t=\bar\mu_0+\int_0^t\frac 1N\sum_{j=1}^Nb(X^{j}_s,X^0_s)ds+B^0_t=\bar\mu_0+\frac 1N\sum_{i=1}^N(X^i_t-(X^i_0+B^i_t))
,\; t\in[0,T].
\end{align*}
Hence, by weak uniqueness for the  particle system \eqref{particleCN}, $(X^{0},\cdots,X^k)$ is distributed according to $Q^{N,k}={\mathcal L}(X^{0,N},X^{1,N},\cdots,X^{k,N})$ under ${\mathbb Q}$. 

We estimate $d_{\rm TV}(Q^{N,k},P^k)$ by reproducing the proof of \cite{Jab19}. We detail it in our specific setting for the sake of completeness. 
The process $Z^N$ satisfies
$$
Z^N_t = 1+\int_0^t Z^N_s \Delta^N_s \cdot \Bigl( \sum_{i=0}^N \ep^i dW^i_s\Bigr).
$$
Fix $M\in \N$ large and set, for $m=0, \dots, M$,  $t_m:= \frac{T m}{M}$. We define ${\mathbb Q}_m$ by 
$$
\frac{d {\mathbb Q}_m}{d{\mathbb P}}= \frac{Z^N_T}{Z^N_{t_m}}= \exp\bigg(\int_{t_m}^T\Delta^N_s\cdot \Bigl( \sum_{i=0}^N \ep ^idW^i_s\Bigr)-\frac{N+1}2 \int_{t_m}^T\big|\Delta^N_s\big|^2ds\bigg).
$$
Note that 
$$
\frac{d {\mathbb Q}_m}{d{\mathbb Q}_{m+1}} = \frac{Z^N_{t_{m+1}}}{Z^N_{t_m}}.
$$
For a fixed $k\in \{1, \dots, N\}$, we also let $Q^{N,k}_m$ be the law of $(X^0, \dots, X^k)$ under ${\mathbb Q}_m$. Then $Q^{N,k}_M=P^k$ and $Q^{N,k}_0=Q^{N,k}$. 
For each $m$, 
$$
\frac{d Q^{N,k}_m}{dQ^{N,k}_{m+1}} = \E_{{\mathbb Q}_{m+1}}\left[  \frac{Z^N_{t_{m+1}}}{Z^N_{t_m}} \ \Bigl|\ X^0, \dots, X^k\right] .
$$
Thus 
$$
d_{\rm TV}(Q^{N,k}_m,Q^{N,k}_{m+1}) =  \E_{{\mathbb Q}_{m+1}} \left[ \left| \E_{{\mathbb Q}_{m+1}}\left[  \frac{Z^N_{t_{m+1}}}{Z^N_{t_m}} -1\ \Bigl|\ X^0, \dots, X^k\right]\right|\right].
$$
If we set 
$$
Z^N_{s,t} = \frac{Z^N_t}{Z^N_s}, 
$$
we have 
$$
\frac{Z^N_{t_{m+1}}}{Z^N_{t_m}} = Z^N_{t_m,t_{m+1}} =1+  \int_{t_m}^{t_{m+1}} Z^N_{t_m,r}\Delta^N_r \cdot \sum_{i=0}^N \ep_i dW^i_r.
$$
Under ${\mathbb Q}_{m+1}$, which coincides with ${\mathbb P}$ on $\sigma((X^0_s,X^1_s,\cdots,X^N_s)_{s\in[0,t_{m+1}]})$, the $(W^i)_{i\geq k+1}$ are independent of $X^0, \dots, X^k$ on $[0,t_{m+1}]$. Therefore 
\begin{align*}
d_{\rm TV}(Q^{N,k}_m,Q^{N,k}_{m+1}) & =  \E_{{\mathbb Q}_{m+1}} \left[ \left|  \E_{{\mathbb Q}_{m+1}}\left[ \int_{t_m}^{t_{m+1}} Z^N_{t_m,r}\Delta^N_r \cdot \sum_{i=0}^k \ep_i dW^i_r \ \Biggl|\ X^0, \dots, X^k\right]   \right|\right]\\
& \leq   \E_{{\mathbb Q}_{m+1}} \left[ \left| \int_{t_m}^{t_{m+1}} Z^N_{t_m,r}\Delta^N_r \cdot \sum_{i=0}^k \ep_i dW^i_r \right|   \right]\\
& \leq  (k+1)^{1/2} \left( \E_{{\mathbb Q}_{m+1}} \left[ \int_{t_m}^{t_{m+1}} (Z^N_{t_m,r})^2| \Delta^N_r|^2dr \right]\right)^{1/2}\\
& = (k+1)^{1/2} \left( \E_{{\mathbb P}} \left[ \int_{t_m}^{t_{m+1}} (Z^N_{t_m,r})^2| \Delta^N_r|^2dr \right]\right)^{1/2},
\end{align*}
since ${\mathbb P}$ and ${\mathbb Q}_{m+1}$ coincide on $\sigma((X^0_s,X^1_s,\cdots,X^N_s)_{s\in[0,t_{m+1}]})$. We get   
\begin{align*}
\E_{{\mathbb P}} \left[ \int_{t_m}^{t_{m+1}} (Z^N_{t_m,r})^2| \Delta^N_r|^2dr \right]  &\leq \E_{{\mathbb P}} \left[ \sup_{r\in [t_m, t_{m+1}]}(Z^N_{t_m,r})^2 \int_{t_m}^{t_{m+1}} | \Delta^N_r|^2dr \right] \\ 
& \leq \left( \E_{{\mathbb P}} \left[ \sup_{r\in [t_m, t_{m+1}]}(Z^N_{t_m,r})^{4}\right]\right)^{1/2} \left( \E_{{\mathbb P}}\left[ \left(\int_{t_m}^{t_{m+1}} | \Delta^N_r|^{2}dr\right)^{2} \right]\right)^{1/2}.
\end{align*}
Using Doob's inequality and the fact that $(Z^N_{t_m,r})_{r\geq t_m}$ is a ${\mathbb P}-$martingale we obtain therefore 
\begin{align*}
\E_{{\mathbb P}} \left[ \int_{t_m}^{t_{m+1}} (Z^N_{t_m,r})^2| \Delta^N_r|^2dr \right] 
& \leq \frac{16}{9} \left( \E_{{\mathbb P}} \left[ (Z^N_{t_m,t_{m+1}})^{4}\right]\right)^{1/2}  \left( \E_{{\mathbb P}}\left[ \left(\int_{t_m}^{t_{m+1}} | \Delta^N_r|^{2}dr\right)^{2} \right]\right)^{1/2}.
\end{align*}
We prove below that the last term in the right-hand size can be bounded as follows: 
\begin{align}\label{kjhazebzrdg}
\E_{{\mathbb P}}\left[\left(\int_r^t\big|\Delta^N_s\big|^{2}ds\right)^2\right] &  \leq C\left(\frac{\|b\|_{\infty}^2(t-r)}{N}\right)^2,
\end{align}
while the middle one can be estimated by:
\be\label{prop31jabir}
 \E_{{\mathbb P}} \left[ (Z^N_{t_m,t_{m+1}})^{4}\right] \leq \frac{C}{1- C   \|b\|_{\infty}(TM^{-1})^{1/2}}, 
 \ee
 provided $M$ is large enough. Postponing  the proof of \eqref{kjhazebzrdg} and \eqref{prop31jabir}, let us conclude the ongoing argument. 
 Choosing $M$ large enough (independently of $N$), we obtain 
 \begin{align*}
d_{\rm TV}(Q^{N,k}_m,Q^{N,k}_{m+1})  & \leq  \frac{C(k+1)^{1/2}}{(1- C   \|b\|_{\infty}(TM^{-1})^{1/2})^{1/4}}\left(\frac{\|b\|_{\infty}^2T}{MN}\right)^{1/2}.
\end{align*} 
Thus 
 \begin{align*}
d_{\rm TV}(P^k,Q^{N,k})=d_{\rm TV}(Q^{N,k}_0,Q^{N,k}_M)\le \sum_{m=0}^{M-1}d_{\rm TV}(Q^{N,k}_m,Q^{N,k}_{m+1})  \leq C\frac{(k+1)^{1/2}}{N^{1/2}}, 
\end{align*}
where this last constant $C$ depends on $\|b\|_\infty$, $T$ and $d$ only. 
\end{proof} 

To complete the proof of the theorem, it remains to check that \eqref{kjhazebzrdg} and \eqref{prop31jabir} hold. 

\begin{proof}[Proof of \eqref{kjhazebzrdg}] We reproduce here \cite[Proposition 4.2]{Jab19}. Note that, as the $X^j$ are independent given $X^0$ and $b$ is bounded, we have, by Hoeffding's inequality combined with \cite[Theorem 2.1]{BLM16}, that, for any integer $q\geq 2$,  
\be\label{jabir1}
\forall s\in[0,T],\;\E_{{\mathbb P}}\left[\bigg|\frac 1N\sum_{j=1}^Nb(X^{j}_s,X^0_s)
  -\E_{{\mathbb P}}[b(X^{1}_s,X^0_s)|X^0]\bigg|^{2q}\bigg|X^0\right]\le q!\left(\frac{4d\|b\|_{\infty}^2}{N}\right)^q.
  \ee
As a consequence,
$$\forall s\in[0,T],\;\E_{{\mathbb P}}\left[\bigg|\frac 1N\sum_{j=1}^Nb(X^{j}_s,X^0_s)
  -\E_{{\mathbb P}}[b(X^{1}_s,X^0_s)|X^0]\bigg|^{2q}\right]\le q!\left(\frac{4d\|b\|_{\infty}^2}{N}\right)^q$$
and, for any $0\le r\le t\le T$, 
\begin{align}\label{skfhjldkgnf}
\E_{{\mathbb P}}\left[\left(\int_r^t\big|\Delta^N_s\big|^{2}ds\right)^q\right] & = 
 \;\E_{{\mathbb P}}\left[\left(\int_r^t\bigg|\frac 1N\sum_{j=1}^Nb(X^{j}_s,X^0_s)
    -\E_{{\mathbb P}}[b(X^{1}_s,X^0_s)|X^0]\bigg|^{2}ds\right)^q\right] \notag\\
    & \leq q!\left(\frac{4d\|b\|_{\infty}^2(t-r)}{N}\right)^q.
\end{align}
The case $q=2$ yields \eqref{kjhazebzrdg}. 
\end{proof}

\begin{proof}[Proof of \eqref{prop31jabir}] We reproduce here the proof \cite[Proposition 3.1]{Jab19} in our setting. We have, for any $p\geq 1$,
\begin{align*}
\E_{{\mathbb P}} \left[ (Z^N_{t_m,t_{m+1}})^{2p}\right] & = 
\E_{{\mathbb P}}\left[  \exp\bigg(2p\int_{t_m}^{t_{m+1}}\Delta^N_s\cdot \Bigl( \sum_{i=0}^N \ep ^idW^i_s\Bigr)-p(N+1) \int_{t_m}^{t_{m+1}}\big|\Delta^N_s\big|^2ds\bigg)\right] \\
& \leq \E_{{\mathbb P}}\left[  \exp\bigg(2p\int_{t_m}^{t_{m+1}}\Delta^N_s\cdot \Bigl( \sum_{i=0}^N \ep ^idW^i_s\Bigr)\bigg)\right] 
\end{align*}
By Taylor expansion of the exponential and Cauchy-Schwarz inequality, we get 
\begin{align*}
 \E_{{\mathbb P}}\left[  \exp\bigg(2p\int_{t_m}^{t_{m+1}}\Delta^N_s\cdot \Bigl( \sum_{i=0}^N \ep ^idW^i_s\Bigr)\bigg)\right] 
 &\leq \sum_{k\geq 0} \frac{(2p)^k}{k!}  \E_{{\mathbb P}}\left[  \bigg(\int_{t_m}^{t_{m+1}}\Delta^N_s\cdot \Bigl( \sum_{i=0}^N \ep ^idW^i_s\Bigr)\bigg)^k\right]\\
 & \leq\sum_{k\geq 0} \frac{(2p)^k}{k!}  \bigg( \E_{{\mathbb P}}\left[  \bigg(\int_{t_m}^{t_{m+1}}\Delta^N_s\cdot \Bigl( \sum_{i=0}^N \ep ^idW^i_s\Bigr)\bigg)^{2k}\right] \bigg)^{1/2} .
\end{align*}
By the Burkholder-Davis-Gundy inequality (see \cite[Remark 2]{CaKr} for the constant) and then \eqref{skfhjldkgnf}, we have 
\begin{align*}
\E_{{\mathbb P}}\left[  \bigg(\int_{t_m}^{t_{m+1}}\Delta^N_s\cdot \Bigl( \sum_{i=0}^N \ep ^idW^i_s\Bigr)\bigg)^{2k}\right]
& \leq 2^{2k}(2k)^k  (N+1)^k \E_{{\mathbb P}}\left[ \bigg(\int_{t_m}^{t_{m+1}}|\Delta^N_s|^2ds\bigg)^{k}\right] \\
& \leq  2^{2k}(2k)^k  (N+1)^k k!\left(\frac{4d\|b\|_{\infty}^2T}{MN}\right)^k.
\end{align*}
Plugging these inequalities together we find 
\begin{align*}
\E_{{\mathbb P}} \left[ (Z^N_{t_m,t_{m+1}})^{2p}\right] &\leq  \sum_{k\geq 0} \frac{(2p)^k}{k!}  2^{k}(2k)^{k/2}  (N+1)^{k/2} (k!)^{1/2} \left(\frac{4d\|b\|_{\infty}^2T}{MN}\right)^{k/2}\\ 
&\leq \sum_{k\geq 0} \frac{k^{k/2}}{(k!)^{1/2}}  \bigg( \frac{(N+1) (2p)^2 32 d  \|b\|_{\infty}^2T}{MN}\bigg)^{k/2},
\end{align*}
where, by the Stirling formula, $k! \geq C^{-1} (2\pi k)^{1/2} (k/e)^k$, so that 
\begin{align*}
\E_{{\mathbb P}} \left[ (Z^N_{t_m,t_{m+1}})^{2p}\right] &\leq  C \sum_{k\geq 0} \bigg( C \frac{ p^2  \|b\|_{\infty}^2T}{M}\bigg)^{k/2},
\end{align*}
where $C$ depends on dimension only. So for $M$ large enough, this quantity if finite and 
\begin{align*}
\E_{{\mathbb P}} \left[ (Z^N_{t_m,t_{m+1}})^{2p}\right] &\leq  \frac{C}{1- C p  \|b\|_{\infty}(TM^{-1})^{1/2}}. 
\end{align*}
\end{proof}

\appendix

\section{An example  without solution}\label{sec.appen}

We consider the problem 
\be\label{qekjsd}
X_t = X_0 + \int_0^t b(X_s,{\rm Var}(X_s|W^0))ds + W^0_t, 
\ee
where $W^0$ is a standard $1$-dimensional Brownian motion,  $X_0$ is independent of $W^0$ and $${\rm Var}(X_t|W^0)= \E[(X_t)^2|W^0]-\E[X_t|W^0]^2.$$ Typically, $b:\R \times \R\to \R$ is Lipschitz in the first variable uniformly in the second one, and globally bounded. Note that this structure is very close to the one discussed in Section \ref{sec.1}, but with the conditional expectation replaced by the conditional variance. 

We assume that $X_0= x^1_0{\bf 1}_A + x^2_0{\bf 1}_{A^c}$, where $x^1_0,x^2_0\in\R$ and the event $A$ is independent of $W^0$ and of probability $1/2$ (to fix the ideas). Let us assume that $X$ is a strong solution to \eqref{qekjsd} and let us set $\tilde X_t= X_t-W^0_t$ and 
$$
\beta_t = {\rm Var}(X_t|W^0)= {\rm Var}(\tilde X_t|W^0).
$$
Then 
$$
\tilde X_t= X_0+ \int_0^t b(\tilde X_s+W^0_s, \beta_s)ds,  
$$
which, given $\beta_s$ and in view of the regularity of $b$, has a unique solution. As $\beta$ is adapted  to the filtration generated by $W^0$ (and thus is independent of $A$), this solution is  of the form $\tilde X_t= \tilde X^1_t{\bf 1}_A+ \tilde X^2_t{\bf 1}_{A^c}$, where, for $j=1,2$, 
$$
 \tilde X^j_t = x^j_0 + \int_0^t b(\tilde X^j_s+ W^0_s,\beta_s)ds, 
 $$
and where $\tilde X^1$ and $\tilde X^2$ are adapted to the filtration generated by $W^0$. Thus, setting $X^j_t= \tilde X^j_t+W^0_t$, 
 $$
 \beta_t = {\rm Var}(\tilde X_s|W^0) = (\tilde X^1_t-\tilde X^2_t)^2/4 =  ( X^1_t- X^2_t)^2/4. 
 $$
  Therefore the initial McKean-Vlasov equation is equivalent to the system of SDEs: 
  $$
  X^j_t  = x^j_0+ \int_0^t b( X^j_s, ( X^1_s- X^2_s)^2/4) ds + W^0_t, \qquad j=1,2,
  $$
  for which a solution does not necessarily exist. This is for instance  the case if $x^1_0= 1$ and $x^2_0=-1$ and if $b=b(y_1,y_2)$ is discontinuous at $y_2= 1$, with 
  $$
  b(y_1,y_2)=-{\rm sign}(y_2-1)\phi(y_1),
   $$
where the Lipschitz continuous and bounded map $\phi:\R\to \R$ is such that $\phi(y_1)= 1$ if $y_1\geq 1/2$ and $\phi(y_1)= -1$ if $y_1\leq -1/2$. Indeed, let $[0,\tau]$ (with $\tau>0$ a.s.) be a random interval  on which $X^1_t\geq 1/2$ and $X^2_t\leq -1/2$ and let us set $D_t= X^1_t-X^2_t-2$. Then on $[0,\tau]$, we have ${\rm sign}((\tilde X^1_t-\tilde X^2_t)^2/4-1)={\rm sign}(D_t)$ and 
 $$
 dD_t = -2{\rm sign}(D_t)dt,  \qquad D_0=0,
 $$
 equation which has no solution.

\end{document}